%% file: gluha.tex
\documentclass{amsart}

\usepackage{amsmath,amssymb,amscd,bbm}
\usepackage[mathscr]{eucal}
\usepackage[all]{xy}
\usepackage[utf8]{inputenc}

 \newtheorem{thm}{Theorem}[section]
 \newtheorem{cor}[thm]{Corollary}
 \newtheorem{lem}[thm]{Lemma}
  \newtheorem{lemdef}[thm]{Lemma and Definition}
 \newtheorem{prop}[thm]{Proposition}
 \theoremstyle{definition}
 
 \theoremstyle{remark}
 \newtheorem{rem}[thm]{Remark}
 \newtheorem{rems}[thm]{Remarks}
 
 \newtheorem*{ex}{Example}
 \numberwithin{equation}{section}

\usepackage{depsymb}
\input{depmakro}

\usepackage{hyperref}

\begin{document}
%-------------------------------------------------------------------------
% editorial commands: to be inserted by the editorial office
%
%\firstpage{1}
%\volume{228}
%\Copyrightyear{2004}
%\DOI{003-0001}
%
%
%\seriesextra{Just an add-on}
%\seriesextraline{This is the Concrete Title of this Book\br H.E. R and S.T.C. W, Eds.}
%
% for journals:
%
%\firstpage{1}
%\issuenumber{1}
%\Volumeandyear{1 (2004)}
%\Copyrightyear{2004}
%\DOI{003-xxxx-y}
%\Signet
%\commby{inhouse}
%\submitted{March 14, 2003}
%\received{March 16, 2000}
%\revised{June 1, 2000}
%\accepted{July 22, 2000}
%
%
%
%---------------------------------------------------------------------------
%Insert here the title, affiliations and abstract:
%
\title{Asymptotics of Operator Semigroups via the Semigroup at Infinity}
% via the Jacobs--de Leeuw--Glicksberg Splitting Theorem}

%----------Author 1
\author[Jochen Gl\"uck]{Jochen Gl\"uck}
\address{%
Ulm University\\ 
Institute of Applied Analysis\\
Helmholtzstr. 18\\
89081 Ulm, Germany}

\email{jochen.glueck@alumni.uni-ulm.de}

%----------Author 2
\author[Markus Haase]{Markus Haase}

\address{%
Kiel University\\
Mathematisches Seminar\\
Ludewig-Meyn-Str.4\\
42118 Kiel, Germany}

\email{haase@math.uni-kiel.de}

%\thanks{This work was completed with the support of our
%\TeX-pert.}

%----------classification, keywords, date
\subjclass{Primary 47D06; Secondary 20M30, 46B42, 47B34, 47B65, 47D03}

\keywords{Convergence of operator semigroups; Jacobs--de
  Leeuw--Glicksberg theory; positive semigroup representations;
  positive group representations; AM-compact operators; kernel
  operators; triviality of the peripheral point spectrum}

\date{\today}
%----------additions
\dedicatory{Dedicated to Ben de Pagter on the occasion of his 65th birthday.}
%%% ----------------------------------------------------------------------

\begin{abstract}
We systematize and generalize recent results of Gerlach and Glück
on the strong convergence and spectral theory of bounded (positive) operator
semigroups $(T_s)_{s\in S}$ on Banach spaces (lattices).
(Here, $S$ can be an arbitrary commutative semigroup, and no
topological assumptions neither on $S$ nor on its representation 
are required.)
To this aim,  we introduce the ``semigroup at infinity'' and give
useful criteria ensuring that the well-known Jacobs--de
Leeuw--Glicksberg splitting theory can  be applied to it. 

Next, we confine these abstract results to
positive semigroups on Banach lattices with a quasi-interior
point. In that situation, 
the said criteria are intimately linked to so-called
AM-compact operators (which entail kernel operators and compact operators);
and they imply that 
the original semigroup asymptotically embeds into a compact group of
positive invertible operators on an atomic Banach lattice. By means
of  a structure theorem for such group representations (reminiscent
of the Peter--Weyl theorem and its consequences for Banach space
representations of compact groups) we are able to establish quite general
conditions implying the strong convergence of the original semigroup.

Finally, we show how some classical results of Greiner (1982), Davies
(2005), Keicher (2006) and Arendt (2008) and more recent ones
by Gerlach and Glück (2017) are covered and extended through our approach.
\end{abstract}

%%% ----------------------------------------------------------------------
\maketitle
%%% ----------------------------------------------------------------------
%\tableofcontents

\section{Introduction}\label{s.intro}

In this paper we deal with the problem of finding 
useful criteria for the strong convergence of a bounded operator semigroup $T$
on a Banach space $E$, with a special view on the asymptotics of
certain semigroups of positive operators on Banach lattices. 
This problem is
not at all new, but 
we refrain from  even trying to give a list of
relevant literature at this point. (However, cf.{} Section
\ref{s.classical} below.) Rather,
let us stress some features that distinguish our approach from most
others.

Classically, asymptotics of operator semigroups focusses on {\em
  strongly  continuous}  one-parameter semigroups $T=(T_t)_{t\in
  [0,\infty)}$. However, there are important instances 
of operator semigroups which lack strong continuity, e.g., 
the heat semigroup on the space of
bounded continuous functions on $\R$ or on the space of finite measures over
$\R$ or, in an abstract context, dual semigroups of 
$C_0$-semigroups on non-reflexive spaces. (See the recent paper \cite{KunzeNL}
for a more involved concrete example.) Hence, there is a need 
for results on asymptotics beyond $C_0$-semigroups.

Secondly, besides the ``continuous time'' case just mentioned, there
is an even more fundamental interest in ``discrete time'', i.e., in the asymptotics of the powers
$T^n$ of a single operator $T$. From a systematic point of view, it is
desirable to try to cover both cases at the same time as far as
possible. This is the reason why we 
consider general semigroup representations $(T_s)_{s\in S}$---where
$(S,+)$ is an Abelian semigroup with zero element $0$---%
as bounded operators on a Banach space
without any further topological assumptions (see Section \ref{s.rep}).

It may not come as a surprise that non-trivial results about
asymptotics can be obtained---even in such a general setting---by employing
the so-called Jacobs--deLeeuw--Glicks\-berg (JdLG) theory for compact (Abelian)
semitopological semigroups. In fact, the role
of the JdLG-theory for asymptotics is well-established. Usually,
it is applied to the  semigroup
\[ \calT := \cl\{ T_s \suchthat s\in S\}
\]
(closure in the strong or weak operator topology) and hence rests on 
a ``global'' compactness  requirement for the whole semigroup. This
is appropriate for the abovementioned ``classical'' cases
(powers of a single operator, $C_0$-semigroups)
because there the strong compactness of $\calT$ is necessary for the
convergence of $T$.  However, typical examples
of non-continuous shift semigroups (left shift on $\Ell{\infty}(0,1)$
or $\co(\R_+)$) show that such semigroups may converge strongly to
$0$ without being relatively strongly compact. (The left shift
semigroup on $\co(\R_0)$ is not even eventually relatively strongly compact.)

In order to cover also these more general situations, we introduce the set
\[  \calT_\infty := \bigcap_{t\in S} \cl\{ T_{s{+}t} \suchthat s\in S\}
\]
which we call the {\emdf semigroup at infinity}. It turns out that 
$\calT_\infty$ is a good replacement for $\calT$ {\em if} 
$\calT_\infty$ is strongly compact and not empty. In particular, its
minimal projection, $P_\infty$, satisfies
\[ P_\infty x = 0 \quad \iff\quad \lim_{s\in S} T_s x = 0 \qquad (x\in
E)
\]
(see Theorem \ref{rep.t.Pinfty}).
Not surprisingly, in the mentioned ``classical''
cases the condition that $\calT_\infty$ be strongly compact and not empty 
is actually equivalent to relative strong
compactness of  the original semigroup (cf. Remark \ref{rep.r.sac}),  
and hence the use of $\calT_\infty$ is then---in
some sense---unnecessary. In a general appoach going beyond
these classical cases, however,  it is our means of choice.

\medskip
In a next step (Section \ref{s.abs}), we describe a convenient
set-up that warrants the crucial property (i.e.: $\calT_\infty$ is 
non-empty and compact). To this end, 
the new notion of {\em quasi-compactness
  relative to a subspace} is introduced. This condition generalizes the
traditional notion of quasi-compactness of a semigroup and plays a
central role in our first main result (Theorem \ref{abs.t.main}).

\medskip
After these completely general considerations, and from then on until
the end of the paper, we confine our attention to
{\em  positive} operator semigroups on Banach {\em lattices}. 
Our second  main result,
Theorem \ref{psg.t.main}, appears to be a mere 
instantiation of Theorem  \ref{abs.t.main} in such a setting.
However, the theorem gains its
significance from the fact that the
required quasi-compactness condition is intimately linked to
the well-known property of {\em AM-compactness} of positive
operators which, in turn, occurs frequently when dealing with
``concrete'' positive semigroups arising in evolution equations and
stochastics (see Appendix \ref{app.int}).

\medskip
 
The main thrust of Theorem \ref{psg.t.main} is that it reduces 
the study 
of the asymptotic properties of certain positive semigroups 
to the following special case: $E$ is an atomic
Banach lattice and $T$ embeds into a strongly compact group 
of positive invertible operators thereon. Hence, in the subsequent Section \ref{s.grp}
we analyze this situation thoroughly and establish  a structure theorem
 which is reminiscent of the Peter--Weyl theorem
and its consequences for Banach space representations of compact
groups (Theorem \ref{grp.t.structure}).

\medskip
Putting the pieces together, in Section \ref{s.conv} we formulate
several  consequences regarding the asymptotic (and
spectral-theoretic)  properties of a positive semigroup $T =
(T_s)_{s\in S}$  satisfying the
conditions of Theorem \ref{psg.t.main}. 
Particularly important here is the fact that 
we can identify (in Theorem \ref{conv.t.main1}) 
two intrinsic properties of the semigroup $S$ that imply
the convergence of $T$: an algebraic one (essential divisibility of
$S$) and a topological one (under the condition that
 $S$ is topological and $T$ is
continuous). Interestingly enough, both are applicable in the case $S=
\R_+$, but none of them in the case $S= \N_0$. 

\medskip
Finally, in Section \ref{s.classical} we review the pre-history of the
problem and show how the results obtained so far by other people
relate to (or are covered by) our findings. Interestingly, in this
pre-history the spectral--theoretic results (``triviality of the
point spectrum'') have taken a much more prominent role than the
asymptotic results. We end the paper with a new and unifying result in
this direction (Theorem \ref{clas.t.gluha}).

\subsection*{Relevance and Relation to the  Work of Others}

The present work can be understood as a continuation and
further development of the recent paper
\cite{GerlachConvPOS} by M.~Gerlach and the first
author. Many 
ideas in the present paper can already be found in \cite{GerlachConvPOS}, like that 
one can go beyond strong continuity by combining the JdLG-theory with the
concept of AM-compactness; or that under AM-compactness conditions
a purely algebraic property
(divisibility) of the underlying semigroup suffices to guarantee the
strong convergence of the representation.

However, we surpass our reference in many
respects:
\begin{aufziii}
\item A general Banach space principle (Theorem \ref{abs.t.main}) is
  established and identified as the theoretical core which
  underlies the results of \cite{GerlachConvPOS}. This principle, which is based
  on the new notion of ``quasi-compactness relative to a subspace''
and on our systematic study of the ``semigroup at infinity'' (Theorem \ref{rep.t.Pinfty}),
  has potential applications in the asymptotic theory of semigroups
  without any positivity assumptions.

\item As a consequence of 1), the
  two main results from \cite{GerlachConvPOS}, Theorem 3.7 and
  Theorem 3.11, are now unified. Moreover, our results hold without
requiring the semigroup to have a quasi-interior fixed point.

\item A general structure theorem for representations
of compact groups on atomic Banach lattices (Theorem
\ref{grp.t.structure}) is established. This result is auxiliary to---but
actually completely independent of---our principal enterprise, the asymptotics of operator semigroups. 
From a different viewpoint, it is a contribution to the theory of
positive group representations as promoted by Marcel de Jeu (Leiden)
and his collaborators.

\item A new spectral-theoretic result (Theorem \ref{clas.t.gluha})
  about the properties of unimodular eigenvalues  is establied. 
\end{aufziii}

\medskip
To understand the relevance of the results obtained in this paper, one best looks into the
pre-history of its predecessor \cite{GerlachConvPOS}. We decided
to place such a historical narrative {\em after} our systematic
considerations, in Section \ref{s.classical}.
That gives us the possibility to then refer freely to the results proven
before, and to explain in detail their relation to the results obtained
earlier by other people.

\subsection*{Notation and Terminology}

We use the letters $E,F, \dots$ generically to denote Banach spaces or
Banach lattices over the scalar field $\K \in \{ \R , \C\}$.
The space of bounded linear operators is denoted by
$\BL(E;F)$, and $\BL(E)$ if $E= F$; the space of compact operators
is $\CO(E;F)$, and $\CO(E)$ if $E= F$.
 Frequently, we shall endow
$\BL(E;F)$  with  the strong operator
  topology (sot). To indicate this we use terms like ``sot-closed'',
``sot-compact'' or speak of ``strongly closed'' or ``strongly  compact''
sets etc. A similar convention applies when the weak operator
topology (wot) is considered. Whereas for a set $A\subseteq E$ the set
$\cl_{\sigma}(A)$ is the closure of $A$ in the weak ($= \sigma(E;E')$)
topology on $E$, the sot-closure and the wot-closure of a set
$M \subseteq \BL(E;F)$ 
are denoted by
\[ \cl_{\sot}(M)\quad \text{and}\quad \cl_{\wot}(M),
\]
respectively.   We shall frequently use the following auxiliary
result, see \cite[Corollary~A.5]{Engel2000}.

\begin{lem}\label{intro.l.sotcp}
Let $E$ be a Banach space. Then a bounded subset $M \subseteq \BL(E)$ is
relatively strongly (weakly) compact if and only if the orbit
\[ Mx := \{ Tx \suchthat T \in M\}
\]
is relatively (weakly) compact for all $x$ from a dense subset of $E$. 
\end{lem}

The set $\BL(E)$
is a semigroup with respect to operator multiplication. Operator
multiplication
is sot- and wot-separately continuous, and it is sot-simultaneously
continuous on norm-bounded sets. 

For the definition of a semigroup
as well as for some elementary definitions and results from 
algebraic semigroup theory, see Appendix \ref{app.sgp}.

We shall freely use standard results and notation
from  the theory of Banach
lattices, with \cite{Schaefer1974} and  \cite{Meyer-Nieberg1991}
being our main references. If $E$ is a Banach lattice the set $E_+:=\{
x\in E \suchthat x\ge 0\}$ is its cone of positive elements. In some
proofs we confine tacitly to real Banach lattices, but there should be no
difficulty to extend the arguments to the complex case.

\section{Representations of  Abelian Semigroups}\label{s.rep}

Throughout the article, 
$S$ is an Abelian semigroup (written additively) containing a neutral
element $0$.

Observe that for each 
$s\in S$ the set
\[ s{+}S := \{ s+r \suchthat r \in S \}  \subseteq S
\]  
is a subsemigroup of $S$.  We turn
$S$ into a directed set by letting
\[ s \le t \quad \defiff \quad t \in s{+}S \quad \iff\quad
t{+}S \subseteq s{+}S.
\]
For limits of nets $(x_s)_{s\in S}$ with respect to this direction,
the notation  $\lim_{s\in S} x_s$ is used. Note that $0 \le s$ for all $s\in S$.

\begin{ex}
Observe that in the cases $S= \Z_+$ and $S= \R_+$ the so-defined
direction and the associated notion of limit coincides 
with the usual one.
\end{ex}

\medskip

\noindent
A {\emdf representation} of  $S$ on a Banach space $E$ is any mapping
$T: S \to \BL(E)$ satisfying
\[ T(0) = \Id \quad \text{and}\quad   T(s+t) = T(s)T(t) \qquad (t,s\in S). 
\]
In place of $T(s)$ we also use index notation $T_s$, and
often call $T= (T_s)_{s\in S}$ an {\emdf operator semigroup} (over
$S$ on $E$).  The {\emdf fixed  space} of the representation $T$ is
\[ \fix(T) := \bigcap_{s\in S} \ker(T_s - \Id) = 
\{ x\in E \suchthat T_s x = x \,\,\text{for all $s\in S$}\}.
\]
An operator semigroup $(T_s)_{s\in S}$ is {\emdf bounded} if
\[ M_T := \sup_{s\in S} \norm{T_s} < \infty.
\]
Boundedness has the following useful consequence.

\begin{lem}\label{rep.l.ct0}
Let $T= (T_s)_{s\in S}$ be a bounded operator semigroup on the Banach
space $E$. Then for each vector $x\in E$ the following assertions
are equivalent:
\begin{aufzii}
\item $0 \in \cls{\{T_s x\suchthat s\in S\}}$.
\item $\lim_{s\in S} T_s x = 0$.
\end{aufzii}
\end{lem}

\begin{proof}
Suppose that (i) holds
and $\veps > 0$. Then there is $s\in S$ such
that $\norm{T_{s}x} \le \veps$. But then
\[  \norm{T_t x} \le \veps \,M_T \qquad \text{for all $t \in s{+}S$.}
\]
It follows that $\lim_{s\in S} T_s x = 0$, i.e.,
(ii).  The converse is trivial.
\end{proof}

Given an operator semigroup $T= (T_s)_{s\in S}$ on a Banach space $E$,
a subset $A$ of $E$ is called {\emdf $T$-invariant} if 
$T_s(A) \subseteq A$ for all $s\in S$. A closed, $T$-invariant
subspace $F$ of $E$ gives rise to a {\emdf subrepresentation}
by restricting the operators $T_s$ to $F$. 
Such a subrepresentation 
is called {\emdf finite-dimensional ($d$-dimensional)} ($d\in \N$)
if $F$ is finite-dimensional ($d$-dimensional).

A one-dimensional subrepresentation is given by 
a scalar representation $\lambda: S \to \K$ and a non-zero vector 
$u\in E$ such that 
\[   T_s u = \lambda_s u\qquad \text{ for all $s\in S$.}
\] 
The corresponding mapping $\lambda$ is then called an {\emdf
  eigenvalue} of $T$, and $u$ is called a corresponding 
{\emdf eigenvector}.  Obviously, 
\[ \text{$\lambda$ is constant}\quad \gdw\quad \lambda = \car
\quad\gdw\quad u \in \fix(T).
\] 
An eigenvalue $\lambda = 
(\lambda_s)_{s \in S}$ is called {\emdf unimodular} if $\abs{\lambda_s} = 1$
for each $s \in S$. (So the constant eigenvalue $\car$ is unimodular.)
An unimodular eigenvalue $\lambda$ of $T$ is called a {\emdf torsion eigenvalue}
if there is $m \in \N$ such that $\lambda_s^m = 1$ for all $s\in S$.

\medskip
If $E$ is a Banach lattice, a semigroup $(T_s)_{s\in S}$ on  $E$ is 
called {\emdf positive} if the positive cone $E_+$ is 
$T$-invariant, i.e., if each operator $T_s$ is positive. 
And a positive semigroup 
is called {\emdf irreducible} or said to {\emdf act
irreducibly} on $E$ if $\{0\}$ and $E$ are the only
$T$-invariant closed ideals of $E$. (Recall
that a  subspace $J$ of $E$ is an ideal if
it satisfies:\quad
$x\in E, \, y \in J, \,\,\abs{x} \le \abs{y} \quad \dann \quad x\in
J.$)

\subsection*{The Semigroup at Infinity}

Given an operator semigroup $T = (T_s)_{s\in S}$ we write
\[ T_S := \{ T_s \suchthat s\in S\} \,\,\subseteq \,\, \BL(E)
\]
for its range, which is a subsemigroup of $\BL(E)$. And we abbreviate
\[ \calT := \cl_\sot \{T_s \suchthat s\in S\} \quad \text{and}\quad
\calT_s := \cl_\sot \{T_t \suchthat t \ge s\}\quad (s\in S),
\]
and call
\[ \calT_\infty := \bigcap_{s\in S} \calT_s =  
\bigcap_{s\in S} \cl_\sot \{T_t \suchthat t \ge s\}
\]
the associated {\emdf semigroup at infinity}. In effect,
$\calT_\infty$ is the set of sot-cluster points of the net $(T_s)_{s\in
  S}$.

Note that $\calT_\infty$ is multiplicative and even satifies 
\[ \calT \cdot \calT_\infty \subseteq \calT_\infty.
\]
But it may be empty (in which case it
is, according to our definition in Appendix \ref{app.sgp}, 
not a semigroup\footnote{We apologize for this little
  abuse of terminology.}.)

\subsection*{The JdLG-Splitting Theory}

One of the principal methods to prove strong convergence of a bounded
semigroup is to employ the splitting theory of Jacobs, de Leeuw and
Glicksberg as detailed, e.g., in \cite[Chapter~16]{Eisner2015}.  
Usually, this theory is applied to the semigroup
$\calT$ or to its wot-counterpart $\cl_\wot\{ T_s \suchthat
s\in S\}$. In contrast, we shall apply it to $\calT_\infty$. If 
$\calT_\infty$ is a strongly compact semigroup, the JdLG-theory
tells that  it contains a unique minimal idempotent, which we denote
by $P_\infty$. (Minimality means that $P_\infty \cdot \calT_\infty$ is
a minimal ideal in $\calT_\infty$.) The range of $P_\infty$ is denoted
here by 
\[ E_\infty := \ran(P_\infty).
\]
Observe that $Q T_s = T_sQ$ for each $s\in S$ and each $Q\in
\calT_\infty$.
In particular, $E_\infty$ is $\calT$-invariant.

\begin{thm}\label{rep.t.Pinfty}\label{rep.t.spec}
Let $T =  (T_s)_{s\in S}$
be  a  bounded  operator  semigroup  on  the  Banach
space $E$ such  that  the  associated  semigroup  at  infinity,
$\calT_\infty$, is strongly compact and non-empty.
Then the following additional assertions hold:
\begin{aufzi}
\item $\calT P_\infty = \calT_\infty P_\infty$.
\item $T$ is relatively strongly compact on $E_\infty$, i.e.,
\[ \calG := \cl_\sot\{ T_s\res{E_\infty} \suchthat s\in
  S\}\subseteq \BL(E_\infty)
\]
is  a strongly compact group of invertible operators on $E_\infty$.
Moreover, 
\[ \calG =  \calT\res{E_\infty} := \{ Q \res{E_\infty} \suchthat Q \in \calT\}.
\]

\item For each $x \in E$ the following statements are equivalent:
\begin{aufzii}
\item $P_\infty x= 0$.
\item $0 \in \cl_\sigma\{T_s x \suchthat s\in S\}$.
\item  $\lim_{s\in S} T_s x = 0$. 
\item  $Rx = 0$ for some/all $R\in \calT_\infty$.
\end{aufzii}

\item If $(\lambda_s)_{s\in S}$ is a unimodular eigenvalue of $T$ with
eigenvector $0 \neq x\in E$, then $x\in E_\infty$ and there is a unique eigenvalue
$\mu = (\mu_Q)_{Q\in \calG}$ of $\calG$ such that $\lambda_s =
\mu_{T_s}$ for all $s\in S$.

\item If $\mu = (\mu_Q)_{Q\in \calG}$ is an eigenvalue of $\calG$ on
  $E_\infty$, then $\lambda_s := \mu_{T_s}$\, ($s\in S$) is an
  unimodular eigenvalue of $T$.
\end{aufzi}
(We suppose $\K = \C$ for assertions {\rm d)} and {\rm e)}.)
\end{thm}

\begin{proof}
a)\ Since $\calT \calT_\infty \subseteq \calT_\infty \calT$, we have
$\calT P_\infty = \calT P_\infty P_\infty \subseteq \calT_\infty
P_\infty \subseteq \calT P_\infty$.

\prfnoi
b)\ By a) we have $\calT\res{E_\infty} = \calT_\infty \res{E_\infty}$, and
the latter is a strongly compact group of invertible operators on $E_\infty$
by the JdLG-theory.  Since restriction is a sot-continuous operator from $\BL(E)$ to
$\BL(E_\infty; E)$, $\calT\res{E_\infty} \subseteq \calG$. The
converse inclusion follows from $\calG P_\infty \subseteq \calT$,
which is true because $P_\infty \in \calT$.

\prfnoi
c)\ If $Rx = 0$ for {\em all} $R\in \calT_\infty$, then clearly
(i) holds, and (i) implies that $Rx=0$ for {\em some} $R\in
\calT_\infty$. On the other hand, this latter statement obviously
implies  $0 \in \cls{\{T_s x\suchthat s\in S\}}$, which is equivalent
to $\lim_{s\in S} T_s x = 0$, i.e., (iii). 

If, in turn,  (iii)  holds  and $\veps > 0$  is  fixed,  then  there  is
$s\in S$ such  that $\{ T_tx \suchthat t \ge  s \} \subseteq
\Ball[0,\veps]$.  Hence,  also $\calT_\infty x \subseteq
\Ball[0,\veps]$.  As $\veps > 0$  was  arbitrary, $\calT_\infty x =
\{ 0\}$, i.e., $Rx= 0$  for all $R\in \calT_\infty$. 

Finally, (iii) obviously implies (ii). Conversely,
starting from (ii) we apply $P_\infty$ to obtain 
\[ 0 \in \cl_\sigma\{ T_sP_\infty x \suchthat s\in S\}.
\]
However, by b) the set $\{ T_sP_\infty x \suchthat s\in S\}$
is relatively strongly compact and hence its weak and its strong 
closures must coincide. This yields
\[  0 \in \cls{\{ T_sP_\infty x \suchthat s\in S\}},
\] 
which implies $P_\infty x = P_\infty(P_\infty x) = 0$ by what
we have already shown. 

\prfnoi
d)\ Let $0\neq x\in E$ be an eigenvector for the unimodular eigenvalue  
$(\lambda_s)_{s\in S}$ of $T$. Define $y := x - P_\infty x$. Then 
$P_\infty y = 0$ and hence $T_s y \to 0$. On the other hand, since
$P_\infty$ commutes with every $T_s$, $T_sy = \lambda_s y$ for all
$s\in S$. As $\abs{\lambda_s} = 1$, it follows that $y = 0$ and hence
$x\in E_\infty$.  The remaining statement now follows easily since $\C
x$ is $T$-invariant and $T_S\res{E_\infty}$ is dense in  $\calG$. 

\prfnoi
e)\ is obvious.
\end{proof}

As a corollary we obtain the following
characterization of the strong convergence of a semigroup.

\begin{cor}
For a  bounded  operator  semigroup  $T =  (T_s)_{s\in S}$
on a  Banach   space $E$ the following assertions are equivalent:
\begin{aufzii}
\item  $T$ is strongly convergent;
\item  $\calT_\infty$ is a singleton;
\item  $\calT_\infty$ is non-empty and strongly compact and acts as
  the identity on $E_\infty$;
\item  $\calT_\infty$ is non-empty and strongly compact and $T$
  acts as the identity on $E_\infty$.
\end{aufzii}
In this case:\quad $\lim_{s\in S} T_s = P_\infty$.
\end{cor}

\begin{proof}
(i)$\dann$(ii): If $T$ is strongly convergent with $P := \lim_{s\in S} T_s$ being its limit, then
$\calT_\infty= \{P\}$ is a singleton.

\prfnoi
(ii)$\dann$(iii): If $\calT_\infty=\{P\}$ is a singleton, then it is
clearly non-empty and strongly compact. It follows that $P =
P_\infty$, and hence $\calT_\infty$ acts as the identity on
$E_\infty$.

\prfnoi
(iii)$\dann$(iv):  Suppose that $\calT_\infty$ is non-empty and
strongly compact and acts  as the identity on $E_\infty$. Let
$Q\in \calT_\infty$. Then, by the equivalence 
(i)$\gdw$(iv) in Theorem \ref{rep.t.Pinfty}.c), 
$Q(\Id - P_\infty) = 0$ and hence $Q = QP_\infty = P_\infty$. So it
follows that $\calT_\infty= \{P_\infty\}$. Moreover,  
since $T_s \calT_\infty \subseteq \calT_\infty$ for each $s\in
S$, we obtain $T_sP_\infty = P_\infty$ and hence $T_s = \Id$ on
$E_\infty$ for all $s\in S$.  

\prfnoi
(iv)$\dann$(i):  Suppose that (iv) holds. Then
\[ \lim_{s\in S} T_s = \lim_{s\in S} (T_sP_\infty + T_s(\Id -
P_\infty)) = P_\infty  + \lim_{s\in S} T_s(\Id - P_\infty) = P_\infty
\]
strongly, by the equivalence (i)$\gdw$(iii) of Theorem \ref{rep.t.Pinfty}.c).
\end{proof}

Theorem  \ref{rep.t.Pinfty}  and its corollary yield  
the  following  strategy  to  prove  strong  convergence  of  an
operator semigroup:
\begin{aufziii}   
\item Show that $\calT_\infty$ is non-empty and strongly compact. 
\item Show that  $\calT_\infty$ (or, equivalently,  $T$) acts as the identity on $E_\infty:= \ran(P_\infty)$.
(For this one may employ the additional information 
that $\calT$ acts on $E_\infty$   as a compact group.)
\end{aufziii}

\begin{rem}\label{rep.r.sap}
Suppose that  $T= (T_s)_{s\in S}$ is a strongly relatively compact operator
semigroup on $E$  with minimal idempotent $P \in \calT$. 
Then, of course,  $\calT_\infty$ is non-empty and strongly
compact, and hence a closed ideal of
$\calT$. It follows from the minimality of $P$ in $\calT$ and
$P_\infty$ in $\calT_\infty$ that
\[  P \calT \subseteq P_\infty \calT \subseteq P_\infty \calT_\infty
\subseteq P \calT_\infty \subseteq P \calT.
\]
Hence $P\calT = P_\infty \calT_\infty$, which implies that $P =
P_\infty$. So, in the case that $T$ is relatively strongly compact,
passing to the semigroup at infinity yields the same
JdLG-decomposition of $E$ as working with $\calT$. 
\end{rem}

We end this section with a technical, but useful characterization 
of the property that  $\calT_\infty$
is non-empty and compact.

\begin{prop} \label{rep.p.sac}
For a  bounded  operator  semigroup  $T =  (T_s)_{s\in S}$
on a  Banach   space $E$ the following assertions are equivalent:
\begin{aufzii}
\item  $\calT_\infty$ is non-empty and strongly compact.
\item  Every subnet of $(T_s)_{s \in S}$ has a strongly convergent
  subnet.
\item Every universal subnet of $(T_s)_{s \in S}$ is strongly convergent.
\item  For each $x \in E$ every subnet of $(T_s x)_{s \in S}$ has a
  convergent subnet.
\item  For each $x \in E$ every universal subnet of $(T_s x)_{s \in
    S}$  converges.
\end{aufzii}
If $S$ contains a cofinal sequence, then the above
assertions are also equivalent to:
\begin{aufzii} \setcounter{aufzii}{5}
\item For every $x \in E$ and every cofinal sequence $(s_n)_{n \in \bbN} \subseteq 
  S$, the sequence $(S_{s_n}x)_{n \in \bbN}$ has a convergent subsequence.
\end{aufzii}
\end{prop}

\begin{proof} 
(i)$\dann$(iv): Suppose  that (i) holds and let $x \in E$. The net
$(T_s (\Id - P_\infty)x)_{s \in S}$ converges to $0$ according to Theorem~\ref{rep.t.spec}.c).
On the other hand, the net $(T_s P_\infty x)_{s \in S}$ is contained 
in the compact set $\calT_\infty P_\infty x$ due to
Theorem~\ref{rep.t.spec} a), so each of its subnets has a convergent
subnet. This shows (iv).

\prfnoi
(iv)$\dann$(v): This follows since a universal net with
a convergent subnet must converge.

\prfnoi
(v)$\dann$(iii):  Let
$(T_{s_\alpha})_{\alpha \in I}$ be a universal subnet
of $(T_s)_{s \in S}$. Then for each $x \in E$, 
the net $(T_{s_\alpha}x)_{\alpha \in I}$ is universal and hence, by
(v), convergent. Thus,
$(T_{s_\alpha})_{\alpha \in I}$ is strongly convergent.

\prfnoi
(iii)$\dann$(ii)$\dann$(i) and (iv)$
\dann$(vi) all  follow from 
Theorem \ref{uni.t.cluster}.

\prfnoi
(vi)$\dann$(i): Suppose that $S$ admits a cofinal sequence. 
Then by Theorem \ref{uni.t.cluster} for each $x\in E$ the 
set $C_x := \bigcap_{t\in S} \cl\{T_s x\suchthat s\ge t\}$ is
non-empty and compact. Since $\calT_\infty x \subseteq C_x$,
it follows that $\calT_\infty$ is strongly compact. 
In order
to  see that $\calT_\infty$ is not empty, fix a
cofinal sequence $(s_n)_n$.  By (vi) and since $E$ is metrizable,
it follows that for each $x\in E$ the set $\{ T_{s_n}x \suchthat
n \in \N\}$ is relatively compact. Hence, $\{ T_{s_n} \suchthat
n \in \N\}$ is relatively strongly compact. It follows that
the sequence $(T_{s_n})_n$ has a cluster point, which is a member
of $\calT_\infty$ since $(s_n)_n$ is cofinal.
\end{proof}

\begin{rems}\label{rep.r.sac}%
\begin{aufziii}
\item 
Assertion~(vi) in Proposition~\ref{rep.p.sac} is called {\emdf strong
  asymptotic compactness} in \cite[p.\,2636]{Emelyanov2001}.

\item Proposition~\ref{rep.p.sac} has an interesting
consequence 
in the ``classical'' cases where $S = \N_0$ or $S= \R_+$ and
$T$ is strongly continuous (cf. the Introduction). Namely, in these cases
one can actually dispense with
the semigroup at infinity, because
$\calT_\infty$ is strongly compact and non-empty {\em if and only if}
$\calT$ is strongly compact.  
\end{aufziii}
\end{rems}

In the next section  we shall present another situation in which  
$\calT_\infty$ is non-empty and strongly compact.

\section{The Abstract Main Result}\label{s.abs}

Suppose that 
$E$ and $F$ are Banach spaces such that  $F$ is densely embedded
in $E$:
\[ F\, \stackrel{d}{\hookrightarrow}\, E.
\]
Reference to this embedding is usually suppressed and $F$
is simply regarded as a subspace of $E$. We take the
freedom to consider an operator on $E$ also as an operator from $F$ to
$E$. (This amounts to view $\BL(E) \subseteq \BL(F;E)$ via the
restriction mapping.)

A semigroup $(T_s)_{s\in S}$ on $E$ is called {\emdf
  $F$-to-$E$ quasi-compact}, or {\emdf quasi-compact relatively to $F$},
if for each $\veps > 0$ there is $s\in S$ and 
a compact operator $K: F \to E$ such that
\[
\norm{T_s - K}_{\BL(F;E)} < \veps.
\]
Note that we do not require that $K$ can be
extended to a bounded operator on $E$. In effect,
the condition of being $F$-to-$E$ quasi-compact can be expressed as 
\[ \dist( \{T_s\suchthat s\in S\}, 
\CO(F;E)) = 0,
\]
where ``$\dist$'' refers to the distance induced
by the norm on $\BL(F;E)$.

\begin{thm}\label{abs.t.main}
Let $E$ and $F$ be Banach spaces, 
with $F$ being densely embedded into $E$, and let $(T_s)_{s\in S}$ be
a bounded operator semigroup on $E$ which restricts
to a bounded operator semigroup on $F$ and
is $F$-to-$E$ quasi-compact. Then the following assertions hold:
\begin{aufzi}
\item $\calT_\infty$ is a strongly compact and non-empty.
\item Each element of $\calT_\infty$ is compact as an operator from
  $F$ to $E$.   
\item  $\calT$ acts on $E_\infty$ as a sot-compact group of invertible
operators.
\item For $x\in E$ the following assertions are equivalent:
\begin{aufzii}
\item  $\lim_{s\in S} T_s x= 0$;
\item  $x \in \ker P_\infty$;
\item  $0 \in \cl_\sigma\{ T_sx \suchthat s\in S\}$.
\end{aufzii}
\end{aufzi}
\end{thm}

\begin{proof}
a) and b)\   
By passing to an equivalent norm on $F$ we may suppose that
each $T_s$, $s\in S$,  is a contraction on $F$. Let
$\Ball_F$ and $\Ball_E$ denote the closed unit balls of $E$ and $F$,
respectively. 

Let $\veps > 0$ and choose $s\in S$ and  $K \in \CO(F;E)$ such that
$\norm{T_s- K}_{\BL(F;E)}\le \veps$. Then
\[ T_{t{+}s}(\Ball_F) = T_s T_t(\Ball_F) \subseteq
T_s(\Ball_F) \subseteq K(\Ball_F) + \veps \Ball_E
\]
for each $t \in S$, and therefore
\begin{equation}\label{AM.e.main} 
\calT_{s}(\Ball_F) \subseteq \cls{K(\Ball_F)} +
\veps \Ball_E.
\end{equation}
Now, let $(T_{s_\alpha})_{\alpha}$ be any universal subnet of
$(T_s)_{s\in S}$ (Lemma \ref{uni.l.kelley}) and let $x\in \Ball_F$. Then
the net $(T_{s_\alpha}x)_\alpha$ is a universal net in $E$. Moreover,
\eqref{AM.e.main} 
shows that for each $\veps> 0$ this net has a tail
contained in the $\veps$-neighborhood of some compact set. Hence,
by Lemma \ref{uni.l.cauchy}, 
it is a Cauchy net and thus convergent in $E$. 
Since $F$ is dense in $E$  and $T$ is bounded, 
$(T_{s_\alpha}x)_\alpha$ converges for {\em every} $x\in E$. In other 
words, $(T_{s_\alpha})_\alpha$ is strongly convergent. As its
limit must be a member of $\calT_\infty$, it follows that 
$\calT_\infty \neq \leer$.

It also follows from \eqref{AM.e.main} that 
$\calT_\infty (\Ball_F) \subseteq \cls{K(\Ball_F)} + \veps \Ball_E$.
As $\cls{K(\Ball_F)}$ is compact, it admits a finite $\veps$-mesh.
Hence,  $\calT_\infty (\Ball_F)$ admits a finite $2\veps$-mesh. Since
this works for each $\veps > 0$, $\calT_\infty (\Ball_F)$ is relatively
compact in $E$.

In particular, it follows that $\calT_\infty \subseteq \CO(F;E)$ 
and that for each $x\in F$
the orbit $\calT_\infty x$ is relatively compact in $E$.
Since $T$ is bounded on $E$ and $F$ is dense in $E$, 
$\calT_\infty$ is relatively strongly compact  (Lemma \ref{intro.l.sotcp}).
But $\calT_\infty$ is strongly closed, so it is strongly compact as claimed.

\prfnoi
c)\ and d) follow from a) by Theorem  \ref{rep.t.spec}. 
\end{proof}

\begin{rem}
Theorem \ref{abs.t.main} seems to be new even for $C_0$-semigroups. In that case, by Remark
\ref{rep.r.sac}.b), it follows a posteriori that the $C_0$-semigroup is
relatively compact. 
\end{rem}

In the next sections we shall see that our set-up from above
has a quite natural instantiation in the
context of semigroups of 
positive operators on Banach lattices with a quasi-interior point.

\section{Positive Semigroups and  AM-Compactness}\label{s.psg}\label{s.AM}

From now on, we consider {\em positive} semigroups $T=(T_s)_{s\in S}$
on Banach lattices $E$. The role of $F$ in our abstract setting
from above will be taken by the {\em  principal ideal}
\[ E_y := \{ x\in E \suchthat \text{there is $c \ge 0$ such that 
$\abs{x}\le cy$}\}
\]
for some  $y \in E_+$, endowed with the natural
AM-norm
\[ \norm{x}_y := \inf\{
 c \ge 0 \suchthat \abs{x}\le cy\}.
\]
Since we need that $F= E_y$ is dense in $E$, we have to require
that $y$ is {\em a quasi-interior point} in $E$. 

As we further need  $E_y$-to-$E$ quasi-compactness, 
it is natural to ask which 
operators on $E$ restrict to compact operators  from $E_y$ to
$E$. It turns out that these are precisely the {\em AM-compact}
operators,
i.e., those that map order intervals of $E$ to
relatively compact subsets of $E$, see Lemma \ref{int.t.AM}. 

There are a couple of useful theorems that help to identify 
AM-compact operators. For example, operators
between $\Ell{p}$-spaces induced by positive integral kernel functions
and positive operators that ``factor through $\Ell{\infty}$-spaces''
are AM-compact. (Proofs of these well-known facts are presented in
Appendix \ref{app.int}, see Theorems \ref{int.t.AM} and \ref{int.t.fact-Linfty}.)

\subsection*{The Range of a Positive Projection}

When we apply  Theorem \ref{abs.t.main} to  a semigroup of positive operators,
the resulting projection $P_\infty$ will be positive, too. 
The following
is a useful information about its range.

\begin{lemdef}\label{psg.l.ranP}
Let $E$ be a Banach lattice and let $P$ be a positive projection
on $E$. Define
\[  \norm{x}_P := \norm{P\abs{x}}\qquad (x\in \ran(P)).
\]
Then the following assertions hold:
\begin{aufzi}
\item $\norm{\cdot}_P$ is an equivalent norm on $\ran(P)$.
\item The space $\ran(P)$ is a Banach lattice with respect to the
order induced by $E$ and the norm $\norm{\cdot}_P$. Its modulus
is given by
\[ \abs{x}_P := P\abs{x} \qquad (x\in \ran(P)).
\]
\end{aufzi}
{\rm The Banach lattice $\ran(P)$ endowed with the
norm $\norm{\cdot}_P$ as in a) and b) is
denoted by $[\ran(P)]$.}
\begin{aufzi}\setcounter{aufzi}{2}
\item If $y\in E_+$ then $P(E_y) \subseteq [\ran(P)]_{Py}$.
In particular, if $y\in E_+$
is a quasi-interior point of $E$ then
$Py$ is a quasi-interior point of $[\ran(P)]$. 
\end{aufzi}
\end{lemdef}

\begin{proof}
This is essentially  \cite[Proposition~III.11.5]{Schaefer1974}.
\end{proof}

In order to obtain further insight into the
relation of the closed ideals in $[\ran(P)]$ and in $E$,
we define  for any Banach lattice $E$ the mapping
\[ \Phi(J) := \cl\bigl\{ x\in E \suchthat 
\abs{x}\le y\,\,\text{for some $y\in J_+$}\bigr\}\qquad (J \subseteq E).
\]
If $J_+ = J \cap E_+$ is a cone, then $\Phi(J) = \Phi(J_+)$ is
the smallest closed ideal in $E$ containing $J_+$.

\begin{thm}\label{psg.t.ideals}
Let $E$ be a Banach lattice and $P$ a positive projection on $E$, and
let $\Phi$ be defined as above.  Then the following assertions hold:
\begin{aufzi}
\item If $I$ is a closed $P$-invariant ideal in $E$ then $P(I) = I \cap
  \ran(P)$ is a closed ideal in $[\ran(P)]$. 

\item If $J$ is a closed ideal in $[\ran(P)]$ then $\Phi(J)$ is
  $P$-invariant and the smallest closed  ideal in $E$ containing $J$. 
Moreover, 
\[  J = P(\Phi(J)) = \Phi(J) \cap \ran(P).
\]
\end{aufzi}
\end{thm}

\begin{proof}
a)\  Let $I \subseteq E$ be a closed $P$-invariant
ideal and let $J := I \cap \ran(P)$. Then $J$ is a closed subspace
of $\ran(P)$. And if $x\in \ran(P)$ and $y \in J$ with $\abs{x}_P \le
\abs{y}_P$, it follows that
\[ \abs{x} = \abs{Px}\le P\abs{x} = \abs{x}_P \le \abs{y}_P = P\abs{y}
\in I
\]
by $P$-invariance. Hence, $x\in J$ and therefore $J$ is an ideal in
$[\ran(P)]$. Moreover, again by $P$-invariance,
\[ J = PJ = P(I \cap \ran(P)) \subseteq P(I)  \subseteq
I \cap \ran(P),
\]
and hence $I \cap \ran(P) = J = P(I)$.

\smallskip
\noindent
b)\ Let $J$ be any closed ideal of $[\ran(P)]$. Then $\Phi(J)$
is the smallest  closed ideal in $E$ containing $J$. (In fact, if 
$x\in J$ then $\abs{x} = \abs{Px}\le P \abs{x} = \abs{x}_P \in J_+$,
and hence $J\subseteq \Phi(J)$.) It is also $P$-invariant, for
if $\abs{x} \le y \in J_+$ then  $\abs{Px} \le P\abs{x} \le Py = y$. 
This also shows that $P(\Phi(J)) \subseteq J$, and since $J \subseteq
\Phi(J)$, it follows that  $P(\Phi(J)) = J$. 
\end{proof}

\subsection*{The Main Result for Positive Semigroups}

We are now prepared for our second main theorem.

\begin{thm}\label{psg.t.main}
Let $T= (T_s)_{s\in S}$ be a bounded and positive 
operator semigroup on a Banach lattice $E$
with a quasi-interior point $y \in E_+$. Suppose, in addition,
that $T$ is $E_y$-to-$E$ quasi-compact and
restricts to a bounded semigroup on $E_y$. Then the following
assertions hold:
\begin{aufzi}
\item $\calT_\infty$ is 
  strongly compact and non-empty and consists of AM-compact operators.

\item $[\ran(P_\infty)]$ is an atomic Banach lattice with order continuous
  norm and quasi-interior point $P_\infty y$.

\item  The semigroup $\calT = \cl_\sot\{T_s\suchthat s\in
  S\}$ acts on $[\ran(P_\infty)]$ as
 a compact topological group of positive, invertible operators.

\item If $(T_s)_{s\in S}$ acts irreducibly on $E$, then $\calT$ acts
irreducibly on $[\ran(P_\infty)]$.
\end{aufzi}
\end{thm}

\begin{proof}
a)\ follows from Theorem \ref{abs.t.main} and Lemma \ref{AM.l.AM-char}.

\prfnoi
b)\ Each order interval $J$ of $E_\infty = [\ran(P_\infty)]$ is of the form
$J = J' \cap E_\infty$, where $J'$ is an order interval of $E$. 
Since $P_\infty$ is AM-compact but restricts to the identity on
$E_\infty$, it follows that $J = P_\infty(J) \subseteq P_\infty(J')
$ is relatively compact. By \cite[Theorem~6.1]{Wnuk1999}, this implies that
$E_\infty$ as a Banach lattice is atomic and has order continuous
norm. 

\prfnoi
c)\ This follows again from Theorem \ref{abs.t.main}.   

\prfnoi
d)\ Suppose that $J\neq \{0\}$ is a closed $\calT$-invariant ideal in
$[\ran(P_\infty)]$. Then the set 
\[ \{ x\in E \suchthat  \abs{x}\le y \,\, \text{for some $y \in J_+$}\}
\]
is $T$-invariant, and hence $\Phi(J)$ is a closed $T$-invariant ideal
in $E$ containing $J$. By irreducibility, $\Phi(J) = E$, and hence
$\ran(P_\infty) = P(E) = P(\Phi(J)) = J$ by Theorem \ref{psg.t.ideals}.
\end{proof}

\begin{rems} \label{psg.r.main}
\begin{aufziii}
\item The assumption that $T$ restricts to a bounded semigroup on $E_y$ is 
for instance satisfied if $y$ is a {\emdf sub-fixed point} of $T$,
i.e., if  $T_t y \leq y$ for all $t \in S$. 

\item  Certainly, if $T_s$ is AM-compact for some $s \in S$ then 
$T$ is $E_y$-to-$E$ quasi-compact. 
\end{aufziii}
\end{rems}

We conclude this section with
 the following result, essentially proved by Gerlach and Glück
in  \cite[Lemma~3.12]{GerlachConvPOS}. It shows that 
in certain situations 
it suffices to require merely that  $T_s$ {\em
  dominates} a non-trivial AM-compact operator for some $s\in S$.

\begin{lem} \label{classical.l.dom}
Let $T = (T_s)_{s \in S}$ be a bounded, positive and irreducible
semigroup  on a Banach lattice $E$ with
order continuous norm and a quasi-interior sub-fixed point $y$ of $T$.
Suppose that there are $s \in S$ and an AM-compact operator 
$K\neq 0$ on $E$ with $0 \le K \le T_{s}$. Then $T$ is  $E_y$-to-$E$
quasi-compact and restricts to a bounded semigroup on $E_y$. In
particular,
Theorem \ref{psg.t.main} is applicable.
\end{lem}

\section{Compact Groups of Positive Operators on Atomic Banach Lattices}
\label{s.grp}

In view of our general strategy, Theorem~\ref{psg.t.main} suggests 
to look for criteria implying that a positive group representation on
an atomic Banach lattice with order-continuous norm is trivial. To
this end, we first summarize some known results about atomic Banach
lattices.

\subsection*{Atomic Banach Lattices}

Recall that an {\emdf atom} in a Banach lattice is any element
$0 \neq a\in E$ such that its generated principal ideal, $E_a$, is
one-dimensional: $E_a = \K \cdot a$. We denote by 
\[
A = A_E := \{ a\in E_+ \suchthat \text{$a$ atom},\, \norm{a} =1\}
\]
the set of positive atoms of norm one.
For distinct $a,\:b\in A$ one has
\[  \abs{a- b} = \abs{a + b} = a + b \ge a
\]
and hence $\norm{a-b} \ge \norm{a} = 1$. This shows that $A$ is a discrete set
with respect to the norm topology. 

A Banach lattice $E$ is
called {\emdf atomic}, if $E$ is the smallest band in $E$ that
contains all atoms.  In other words, 
\[ A^d : = \{ x\in E \suchthat \abs{x} \wedge a = 0 \,\,\text{for all
  $a\in A$}\} = \{0\}.
\]
For each $a\in A$ the one-dimensional subspace $E_a = \K a$ is a
projection band, with corresponding band projection $P_A$ given by 
\[ P_a x := \sup [0,x]\cap \R a = \sup\{ t\in\R_+ \suchthat ta \le x\}
\cdot a \qquad (x\in E_+).
\]
(See, e.g. \cite[Thm. 26.4]{Luxemburg1971} and cf.{}  
\cite[Prop. 1.2.11]{Meyer-Nieberg1991}.) The next result is a
consequence of  \cite[p.143, Ex. 7]{Schaefer1974} 
and \cite[Thm. 1.2.10]{Meyer-Nieberg1991}. 
For the convenience of the reader, we give a proof.

\begin{thm}\label{grp.t.atomic-base}
Let $E$ be a Banach lattice and let $A$ be its set of positive
normalized atoms. Then for each finite subset
$F\subseteq A$ the space
\[   \spann(F)  = \bigoplus_{a\in F} \K a 
\]
is a projection band with band projection\quad $\displaystyle P_F = \sum_{a\in F} P_a$.
\quad Suppose, in addition, that  $E$ is atomic. Then 
\beq\label{grp.eq.atomic}  
\Id_E = \sum_{a\in A} P_a
\eeq
as a strongly order convergent series. Each band $B$ in $E$ is generated
(as a band) by $A \cap B$.
\end{thm}

\begin{rem}
There are different notions of ``order convergence'' 
in the literature, see \cite{Abramovich2005}. We employ the definition
found in \cite[Definition~1.1.9 i)]{Meyer-Nieberg1991}. For the case
of \eqref{grp.eq.atomic} this simply means  
\beq\label{grp.eq.atomic-full}    
x = \sup_F \sum_{a\in F} P_a x \qquad \text{for all $x\in E_+$},
\eeq
where the supremum is taken over all finite subsets of $A$. 
\end{rem}

\begin{proof}[Proof of Theorem~\ref{grp.t.atomic-base}]
Fix a  finite set $F \subseteq A$. If $a, b \in F$ with $a \neq b$,
then  $a \wedge b = 0$ and hence $P_a P_b = 0$.  It follows that 
\[ P_F := \sum_{a\in F} P_a
\]
is a projection (and $F$ is a linearly independent set).
Again by the pairwise disjointness of the elements of $F$,
\[ \sum_{a\in F} P_a x = \bigvee_{a\in F} P_a x \le x \qquad (x\in E_+).
\]
This shows that $0 \le P_F \le \Id$, and hence $P_F$ is a band
projection \cite[Lemma 1.2.8]{Meyer-Nieberg1991}. Since, obviously,
$\ran(P_F) = \spann(F)$, the first assertion is proved.

In order to prove \eqref{grp.eq.atomic-full}    
fix $x\in E_+$ and let $y\in E_+$ be such that $y \ge P_Fx$ for all
finite $F\subseteq A$. Then $y\ge P_a x$ for each $a\in A$  and hence
\[ 0 \le x - (x\wedge y)  \le x - P_a x  \perp a
\]
If $E$ is atomic, it follows that $x = x\wedge y$, i.e., $x\le y$. 
This yields \eqref{grp.eq.atomic-full}.

Finally, let $B\subseteq E$ be any band and let $0 \le x\in
B$. Then for each $a\in A$, $P_a x \in B$ (since $0 \le P_a x \le x$
and $B$ is an ideal). Hence, either $P_ax = 0$ or $a \in B$. 
It follows from \eqref{grp.eq.atomic}   that 
$B$ is generated by $A\cap B$.
\end{proof}

With this information at hand we now turn to
the representation theory.

\subsection*{A Structure Theorem}

Let $G\subseteq \BL(E)$ be a group of positive, invertible operators
on $E$.  (In particular, $G$ consists of lattice homomorphisms.)
Then for each $g\in G$ and $a\in A$ the element $g \cdot a \in
E$ must be an atom again. In effect
\beq\label{grp.eq.vphi_g} 
\vphi_g(a) := \norm{g\cdot a}^{-1} (g\cdot a) \in A.
\eeq
It is easy to see that
\beq\label{grp.eq.vphi}
\vphi: 
G \to \Sym(A),\qquad g\mapsto \vphi_g
\eeq
is a group homomorphism from $G$ to the group of all bijections on
$A$. The corresponding action 
\[  G \times A \to A, \qquad (g,a) \mapsto \vphi_g(a)
\]
is called the {\emdf induced action} of $G$ on $A$. 
 For each $a\in A$ the orbit mapping
\[
 G \to A,\qquad g \mapsto \vphi_g(a)
\]
of the induced action 
is continuous (with respect to the strong operator topology on $G$).  
If, in addition, $G$ is strongly compact, then each orbit 
\[ \vphi_G(a) = \{ \vphi_g(a) \suchthat g\in G\}
\]
is finite (since $A$ is discrete).
We denote by  $A/G$ the set of all these orbits. Then 
$A/G$ is a partition of $A$ into finite subsets.

\begin{lem}\label{grp.l.fix}
In the described situation, suppose that $G$ is compact.
Then for  $a\in A$ and $g\in G$:
$g \cdot a = a \,\,\gdw\,\, \vphi_g(a) = a$. 
Furthermore: $g = \Id_E \,\,\gdw \,\,\vphi_g = \id_A$.
\end{lem}

\begin{proof}
Fix $a \in A$ and $g\in G$. If $g \cdot a = a$  then $\vphi_g(a)=
a$, since $\norm{a} = 1$. Conversely, suppose that $\vphi_g(a) = a$. 
Then $g^n \cdot a = \norm{g \cdot a}^n a$ for all $n\in \Z$. By
compactness, $\norm{g \cdot a} = 1$, and hence $g\cdot a = a$ as claimed.
 
\prfnoi
Suppose that $\vphi_g(a)= a$ for all $a\in A$. Then, as we have just
seen,  $g \cdot a = a$ for all $a\in A$. So $g$ leaves all atoms
fixed. Since $g$ acts a lattice isomorphism and hence is order
continuous, it follows from Theorem \ref{grp.t.atomic-base} that $g = \Id_E$.
\end{proof}

We can now prove a theorem that is reminiscent of the Peter--Weyl
structure theorem and its applications to Banach space
representations of compact groups.

\begin{thm}[Structure Theorem]\label{grp.t.structure}
Let $E \neq \{0\}$ be an atomic Banach lattice  and let  
$A$ be its set of positive  normalized atoms. Let
$G\subseteq \BL(E)$ be a strongly compact group of positive invertible
operators on $E$, and let
$A/G$ be the set of orbits of elements
of $A$ under the induced action of $G$ on $A$. Then the following assertions hold:
\begin{aufzi}
\item For each orbit $F\in A/G$ the band $\spann(F)$ is 
  $G$-invariant, the corresponding band projection $P_F$ is
  $G$-intertwining, and $G$ acts irreducibly on $\spann(F)$. 

\item If $B\neq\{0\}$ is a $G$-invariant band in  $E$ on which
$G$ acts irreducibly, then $B = \spann(F)$ for some $F \in A/G$.

\item $I = \sum_{F\in A/G} P_F$ as a strongly  order-convergent
  series.

\item In the  case $\K = \C$, each eigenvalue of $G$ on $E$ is
  torsion.

\item If $G$ acts irreducibly on $E$, then $\dim(E)< \infty$ and 
$G$  has only finitely many eigenvalues. 
\end{aufzi}
\end{thm}

\begin{proof}
a)\  It is obvious that $\spann(F)$ is $G$-invariant and $G$ acts
irreducibly on it. Since $G$ consists of 
lattice automorphisms, also  $\spann(F)^d$ is $G$-invariant, and hence
$P_F$ is $G$-intertwining. 

\prfnoi
b)\ Let $B\neq\{0\}$ be any $G$-invariant band in $E$. Then $B$ is generated
(as a band) by $A \cap B$. By $G$-invariance,
$A \cap B$ is a union of $G$-orbits (for the induced action), i.e., a
union of elements of $A/G$. 
Hence, if $G$ acts irreducibly on $B$, $A\cap B$ must coincide with
precisely  one $G$-orbit of $A$, i.e., $A \cap B \in A/G$.

\prfnoi
c)\ follows from Theorem \ref{grp.t.atomic-base} since $A/G$ is partition of
$A$.

\prfnoi
d)\ Let $\lambda : G \to \C$ be an eigenvalue  of $G$ on $E$ and $0
\neq  x\in E$ a corresponding eigenvector. Then $\lambda$ is a
continuous homomorphism, and since $G$ is compact, $\lambda$ is
unimodular. By c), one must have $y := P_F x \neq 0$ for some $F\in A/G$,
and by a), $y$ is also an eigenvector corresponding to $\lambda$.
 
Let $g \in G$  and $n := \abs{F}$, the length of the (induced) $G$-orbit
$F$. Then  $g^{n!}$ acts (induced) on $F$ as the identity. Hence, by
Lemma \ref{grp.l.fix}, $g^{n!}$ acts (orginally) as the identity on
$\spann(F)$. This yields
\[  y =  g^{n!}y = \lambda_g^{n!}y
\]
and hence $\lambda_g^{n!} = 1$. As $g\in G$ was arbitrary, the eigenvalue $\lambda$
is torsion.

\prfnoi
e)\ If $G$ acts irreducibly on $E$,  then b) tells that 
$E$ is finite dimensional. As
eigenvectors belonging to different eigenvalues have to be linearly
independent, there can only finitely many eigenvalues, as claimed.
\end{proof}

\begin{rem}
For the special case of Banach sequence spaces, Theorem \ref{grp.t.structure}
has been first proved by  de Jeu and Wortel in  \cite[Theorem 5.7]{Jeu2014}.
\end{rem}

We now shall list several criteria 
for the group $G$ in Theorem~\ref{grp.t.structure} 
to be trivial. In Section~\ref{s.conv} we will translate those criteria
into sufficient conditions for the strong convergence of positive operator semigroups.

A positive linear operator $T$ on a Banach lattice $E$ is called 
{\emdf strongly positive} if $Tf$ is a quasi-interior point for every 
non-zero positive vector $f \in E$.

\begin{cor} \label{grp.c.misc}
Let $G$ be a strongly compact group of positive invertible operators
on an atomic Banach lattice $E\neq \{0\}$.
Then each one of the following 
assertions implies that $G = \{\Id_E\}$: 
\begin{aufziii}
\item $G$ is divisible (cf.~Appendix~\ref{app.sgp}).
\item $G$ has no clopen subgroups different from $G$. 
\item $G$ contains a strongly positive operator.
\item Every finite-dimensional $G$-invariant band of $E$ on which $G$ acts irreducibly 
has dimension $\le 1$.
\item $\K = \C$ and  $G$ is Abelian and does not have any non-constant 
torsion eigenvalues.
\end{aufziii}
\end{cor}

\begin{proof}
By Lemma \ref{grp.l.fix} it it suffices to prove 
in each of the mentioned cases that the group homomorphism 
$\vphi$, defined  in \eqref{grp.eq.vphi}, is trivial. We fix $a\in A$
and abbreviate $F := \vphi_G(a)$. 

\prfnoi
1)\ $G$ acts transitively on $F$, which is a finite set.
By a standard result from group theory, each homomorphism from a
divisible group into a finite one must be trivial (see
 \cite[Lemma~2.3 and Proposition~2.4]{GerlachConvPOS} for a proof).
Hence $F = \{a\}$.

\prfnoi
2) The set $H := \{g \in G \suchthat \vphi_g(a) = a\}$
is a clopen subgroup of $G$ (since $A$ is discrete),  so $H = G$.

\prfnoi
3) Suppose that $g\in G$ is strongly positive. Then 
$\vphi_g(a)$ is a quasi-interior point and an atom, hence $\dom(E) =1$.
In particular, $F = \{a\}$.

\prfnoi
4)\ By Theorem \ref{grp.t.structure}, $\spann(F)$  is a 
finite-dimensional $G$-invariant
band of $E$ on which  $G$ acts irreducibly.  Hence $1 \le \abs{F} =
\dim(\spann(F))\le 1$, by assumption. If follows that $F= \{a\}$.  

\prfnoi
5)\  Let $m := \abs{F}$. We may 
consider $G$ as a compact Abelian group of $m \times m$-matrices
acting on $\spann(F) \cong \C^m$.  Since $G$ is commutative, 
it is simultaneously diagonalizable. Each diagonal entry in a
simultaneous diagonalization is an eigenvalue of $G$. 
By Theorem \ref{grp.t.structure} such an eigenvalue is torsion, and hence, by
asumption, trivial. This means that $G$ acts as the identity on $\spann(F)$,
which implies that $F = \{a\}$.
\end{proof}

\section{Convergence of Positive Semigroups} \label{s.conv}

We shall now combine  Theorem~\ref{psg.t.main} with the findings of the
previous section to obtain general results about strong convergence
of positive operator semigroups.  In all results of this section we shall
take the following 
hypotheses (those of Theorem~\ref{psg.t.main}) as a starting point:
\begin{itemize}
\item $E$ is a Banach lattice;
\item $T= (T_s)_{s\in S}$ is a positive and bounded operator semigroup 
on $E$;
\item $T$ restricts to a bounded semigroup on $E_y$ and
is $E_y$-to-$E$ quasi-compact for some quasi-interior point $y\in E_+$.
\end{itemize}
Let us call these our {\emdf standard assumptions} for the remainder
of this paper. The standard assumptions warrant that Theorem \ref{psg.t.main}
is applicable, and we freely make use of this fact in the following.

\subsection*{Spectral-Theoretic Consequences}

We  first draw some spectral-theoretic conclusions.

\begin{thm}\label{conv.t.spec}
Suppose that an operator semigroup $(T_s)_{s\in S}$ on 
a complex Banach lattice $E$ satisfies the standard assumptions.
Then the following assertions hold:
\begin{aufzi}
\item Each unimodular eigenvalue of $T$ is torsion.
\item If $T$ is irreducible, then it has only finitely many unimodular
  eigenvalues.
\item $T$ is strongly convergent if and only if $T$ has no 
 non-constant torsion eigenvalue.   
\end{aufzi}
\end{thm}

\begin{proof}
a)\  By Theorem \ref{rep.t.spec}, 
each unimodular eigenvalue of $T$ is the restriction
of an eigenvalue of 
\[ \calG := \calT\res{E_\infty},
\]
where $E_\infty = [\ran(P_\infty)]$. Since the latter space is atomic and $\calG$ is
compact, Theorem \ref{grp.t.structure} yields that this eigenvalue is torsion.

\prfnoi
b)\ If $T$ acts irreducibly on $E$ then, by Theorem \ref{psg.t.main},  
$\calG$ acts irreducibly on $[\ran(P_\infty)]$. Hence, $\calG$ has
only finitely many unimodular eigenvalues by Theorem \ref{grp.t.structure}. 
By Theorem \ref{rep.t.spec}, each eigenvalue of $T$ is the restriction of 
an eigenvalue of $\calG$. This proves  the claim.

\prfnoi
c)\ The ``if''-part follows from Corollary \ref{grp.c.misc}. For the
``only if''-part we suppose that $T_s \to P$ is strongly
convergent.  Then $P = P_\infty$ and $\calG =
\{\Id_{E_\infty}\}$. Since each unimodular eigenvalue of $T$ is the
restriction of an eigenvalue of $\calG$ (Theorem \ref{rep.t.spec}), 
$T$ has no non-constant unimodular eigenvalues.
\end{proof}

\subsection*{Sufficient Conditions for Convergence}

Apart from the spectral characterization of the previous theorem,
Corollary \ref{grp.c.misc} yields the following sufficient conditions
for the convergence of a positive semigroup.

\begin{thm}\label{conv.t.main1}
Suppose that an operator semigroup $(T_s)_{s\in S}$ on 
a Banach lattice $E$ satisfies the standard assumptions.
In addition, let at least one of the following conditions
be satisfied:
\begin{aufziii}
\item $S$ is essentially divisible (e.g.: $S$ is divisible or generates
  a divisible group; cf.~Appendix~\ref{app.sgp});
\item $S$ carries a topology such that $T$ is 
strongly continuous and the 
only clopen subsemigroup of $S$ containing $0$ is $S$ itself (e.g.: $S$ is connected).
\item $T_s$ is strongly positive for some $s\in S$.
\end{aufziii}
Then $T$ is strongly convergent.
\end{thm}

\begin{proof}
Theorem~\ref{psg.t.main} is applicable, so it suffices to consider the case
that $E\neq \{0\}$ is atomic and has order-continuous norm and that $\calT$ is a compact group of positive invertible operators on $E$. We must
show that $\calT$ acts trivially on $E$.

\prfnoi
1)\ By Corollary \ref{grp.c.misc} it suffices to show that $\calT$ is
divisible. But this follows from a straightforward compactness
argument.

\prfnoi
2)\ Let $\calH$ be any clopen subgroup of $\calT$. Then
$H := \{ s\in S \suchthat T_s \in \calH\}$ is a clopen subsemigroup of
$S$. (Note that $H \neq \leer$ since $\calH \neq \leer$ is open and
$T_S$ is dense in $\calT$.) By hypothesis, $H = S$, so 
$T_S \subseteq \calH$. Since $\calH$ is closed and $T_S$ is dense in
$\calT$, it follows that $\calH= \calT$. Hence, $\calT= \{ \Id_E\}$ 
by Corollary \ref{grp.c.misc}. 

\prfnoi
3)\ If $T_s$ is strongly positive, then $\calT$ contains a strongly 
positive operator, and we conclude with the help of
Corollary \ref{grp.c.misc}. 
\end{proof}

\begin{rems}
\begin{aufziii}
\item Condition 1) is satisfied, in particular, if  
$S =\R_+$ (divisible semigroup), but also 
if $S= \{0\} \cup [1, \infty)$ (not divisible, but generating a
divisible group).  Note that in the latter case, 
the semigroup direction is just a subordering of the natural one,
but does not coincide with it. Nevertheless, the associated
notions of ``limit'' do coincide.

\item Condition 1) is also satisfied when $S = [0, \infty)$, endowed 
with  the semigroup operation $(a,b) \mapsto a \vee b =
\max\{a,b\}$. The semigroup direction coincides with the natural ordering. 
This semigroup is neither divisible nor
does it generate a divisible group (it is not even cancellative). 
However, it is essentially divisible. 

On the other hand, this example is a little artificial, as each
element of $S$ is an idempotent, and hence a representation 
$T= (T_s)_s$ is just a family of projections with decreasing ranges
as $s$ increases. For such semigroups, the question of 
convergence can often be treated by other methods. 

\item The semigroup of {\emdf positive dyadic rationals} is
\[ D_+ := \{0\} \cup \bigl\{ \tfrac{k}{2^n} \suchthat k, n \in \N_0
\bigr\}.
\]
The semigroup direction on $D_+$ coincides with the usual ordering.
It is easy to see that $D_+$ is not essentially divisible. 
If we endow $D_+$ with its natural topology, $D_+$ is not
connected. However, $D$ is the only clopen subsemigroup of $D_+$
containing $0$.
(Actually, apart from $D_+$ itself there is no other
{\em open} subsemigroup of $D_+$ containing $0$.) 

Hence, from Theorem \ref{conv.t.main1} 
it follows that each strongly continuous representation 
of $D_+$ that satisfies the standard hypotheses is strongly convergent. 
Without strong continuity, however, this can fail.
In fact, let $D = D_+ - D_+$ denote the group generated by $D_+$ in the 
real numbers. Then $D$ has a subgroup of index $3$ (namely $3D$), so by the
same construction as in the proof of \cite[Theorem~2.5]{GerlachConvPOS}
we can find a positive and bounded representation of 
$(D,+)$ on the Banach lattice $\bbR^3$. The restriction of this representation
to $(D_+,+)$ satisfies the standard conditions, but it does not converge.

\item With Theorem \ref{conv.t.main1}, Condition 3), we  
generalize a result of Gerlach, cf.{} \cite[Theorem~4.3]{Gerlach2013}.
\end{aufziii}
\end{rems}

\medskip

\subsection*{Lattice Subrepresentations}

A closed linear subspace $F$ of Banach lattice $E$ is called a
{\emdf lattice subspace}, if it is a Banach lattice
with the order induced by $E$ but with respect to an equivalent norm.
A lattice subspace need not be a sublattice. 
(By Theorem \ref{psg.l.ranP}, the range of a positive projection is  always a
lattice subspace.)

Given a representation $T= (T_s)_{s\in S}$ on a Banach lattice $E$,
each  $T$-invariant lattice subspace gives rise to a 
{\emdf lattice subrepresentation}. So the lattice subrepresentations
are those subrepresentations where the underlying space is
a lattice subspace.

\begin{thm} \label{conv.t.latsub}
Suppose that an operator semigroup $T= (T_s)_{s\in S}$ on 
a Banach lattice $E$ satisfies the standard assumptions.
In addition,  suppose that  each finite-dimensional 
lattice subrepresentation of $T$ is   at most one-dimensional.
Then $T$ is strongly convergent.
\end{thm}

\begin{proof}
As in the proof of Theorem \ref{conv.t.main1} 
it suffices to consider the case
that $E\neq \{0\}$ is atomic and that
$\calT$ is a compact group of positive invertible operators on $E$. 
It then follows that each finite-dimensional
$T$-invariant band of $E$ is at most one-dimensional.
Corollary~\ref{grp.c.misc} is applicable and yields the claim.
\end{proof}

\section{Conclusion: Some Classical Theorems Revisited} \label{s.classical}

In this section we start with a little historical survey and end with
demonstrating how  our approach leads to far-reaching generalizations
of the ``classical'' results.

\subsection*{Historical Note}

In 1982, Günther Greiner in the influential paper \cite{Greiner1982a} 
proved the following result as ``Corollary~3.11'':

\begin{thm}[Greiner 1982]\label{clas.t.greiner}
Let $T = (T_s)_{s\ge 0}$ be a positive contraction $C_0$-semigroup on
a space $E = \Ell{p}(\prX)$, $1\le p < \infty$, with the following
properties:
\begin{aufziii}
\item There is a strictly positive $T$-fixed vector;
\item For some $s_0 > 0$ the operator $T_{s_0}$ is a kernel operator.
\end{aufziii}
Then $\lim_{s\to \infty} T_s x$ exists for each $x\in E$.
\end{thm}

For the proof, Greiner employed what has become known as ``Greiner's
$0/2$-law'' (see \cite[Theorem~3.7]{Greiner1982a} and also \cite{Greiner1982})
 and a result
of Axmann from \cite{Axmann1980}. Both results have involved proofs and make
use of the lattice structure on the regular operators on Banach
lattices with order-continuous norm. 
The relevance of Greiner's theorem
derives from the fact that the assumptions can be frequently verified
for semigroups arising in partial differential equations or in
stochastics.

\medskip
For a long time, Greiner's theorem stood somehow isolated within the
asymp\-totic theory of (positive) semigroups.  
The ``revival'' of Greiner's theorem as a theoretical result began
with a paper of Davies \cite{Davies2005} from 2005. Davies showed that 
the peripheral point spectrum of the generator $A$ of a $C_0$-semigroup 
$T$ of positive contractions on a space $E = \Ell{p}(\prX)$, $1\le p <
\infty$, has to be trivial in
the following cases: (1) $\prX$ is countable with the counting measure
and (2) $\prX$ is locally compact and second countable and $T$ has the
Feller property (i.e., each $T_s$ for $s > 0$ maps $E$  into the space of
continuous functions). Case (1) was subsequently generalized by
Keicher in \cite{Keicher2006} to bounded and positive $C_0$-semigroups on
atomic Banach lattices with order-continuous norm, and by Wolff
\cite{Wolff2008} to more general atomic Banach lattices.

Shortly after, Arendt in \cite{Arendt2008} generalized Davies' results
towards the following theorem.

\begin{thm}[Greiner 1982/Arendt 2008]\label{clas.t.arendt}
Let $A$ be the generator of a positive contraction $C_0$-semigroup $T
= (T_s)_{s\ge 0}$ on a space $E = \Ell{p}(\prX)$, $1\le p <\infty$. 
Suppose that  for some $s_0 > 0$ the operator $T_{s_0}$ is a kernel operator.
Then $\Pspec(A) \cap \ui \R \subseteq \{0\}$.
\end{thm}

This result is ``Theorem 3.1'' in Arendt's paper
\cite{Arendt2008}. Interestingly, as observed by Gerlach in
\cite{Gerlach2013}, it already appears in Greiner's 1982 paper, namely 
in the first paragraph of his proof of Theorem \ref{clas.t.greiner}
(i.e., his ``Corollary~3.11''). We
will thus call Theorem \ref{clas.t.arendt} the {\em Greiner--Arendt theorem}. 
Arendt points out that Theorem \ref{clas.t.arendt} implies Davies' result: 
in case (1) every positive operator is a kernel operator, whereas in case (2) the Feller 
property implies that each $T_s$ for $s > 0$ is a kernel
operator. (This follows from Bukhvalov's characterization of kernel
operators, cf.~\cite[Corollary 2.4]{Arendt2008}.) 

Let us briefly sketch Arendt's proof of Theorem \ref{clas.t.arendt}: 
If $f\in E$ is an eigenvector of $A$ for the eigenvalue $\lambda \in \ui \R$,
then $T_s\abs{f} \ge \abs{f}$
for all $s\ge 0$. Since each $T_s$ is a contraction and the norm on
$\Ell{p}$ is strictly monotone, it follows that $\abs{f}$ is a fixed
point. By restricting to the set $\set{ \abs{f} > 0}$ one
can assume that $\abs{f}$ is strictly positive. Next, from the
weak compactness of the order interval $[0, \abs{f}]$ it follows that the semigroup 
is weakly relatively compact. Then the JdLG-theory enters
the scene and reduces to problem to an atomic Banach lattice with
order continuous norm. Finally, Keicher's analysis from \cite{Keicher2006} shows that the
dynamics there must be trivial, and hence $\lambda=0$.

Arendt's paper is remarkable in several respects. First of all, 
his proof of Theorem \ref{clas.t.arendt} employs the JdLG-theory which is 
central also to the more recent  work of Gerlach
and Glück, and to the present paper. Secondly, Arendt recalls
Greiner's Theorem \ref{clas.t.greiner} and  gives a 
proof (building, 
as Greiner did, on  Theorem \ref{clas.t.arendt})
under the additional assumption that the semigroup is
irreducible. (This proof  appears
to be the first complete one in English language, cf. \cite[Remark
4.3]{Arendt2008}.)
Thirdly, Arendt promotes Greiner's  result by illustrating its use with
several concrete examples.

Most remarkable of all, however,  is what is {\em not} written in
\cite{Arendt2008}: namely that 
Greiner's Theorem \ref{clas.t.greiner} almost
directly implies the Greiner--Arendt Theorem \ref{clas.t.arendt}. Indeed, one starts
exactly as in Arendt's proof until one has found the quasi-interior fixed point
$\abs{f}$; then  Greiner's theorem tells that  $\lim_{s\to \infty}
T_sf$ exists, and hence $\lambda = 0$ follows.

\medskip
In the following years the topic was taken up by M.~Gerlach and
J.~Glück. Gerlach \cite{Gerlach2013} discussed Greiner's approach in a general
Banach lattice setting and extended it to semigroups that merely dominate a kernel operator; he also noted
that the dominated kernel operator can be replaced by a compact
operator. In their quest to find a  unifying framework, and stimulated 
by ``Corollary~3.8'' in Keicher's paper \cite{Keicher2006}, 
Gerlach and Glück in \cite{Gerlach2017,GerlachConvPOS} finally
identified ``AM-compactness'' as the right property generalizing the
different cases. Alongside a unification, this led also to a major
simplification, since AM-compactness is  much more
easily shown directly than by  passing through the concept of a kernel
operator. (E.g., it follows directly from Theorem \ref{int.t.fact-Linfty} that a
Feller operator as considered by Davies is AM-compact.)  

Finally, Gerlach and Glück realized that strong continuity of the
semigroup can be dispensed with, since  arguments
requiring time regularity can be replaced by purely
algebraic ones. This led to proofs for most of the above-mentioned results
for semigroups without any time regularity. 

\medskip
Somewhat independently from the above development, Pich\'or and
Rudnicki proved convergence results for Markov semigroups which merely dominate a
non-trivial kernel operator \cite[Theorems~1 and~2]{Pichor2000}. Their
results are closely related to (and earlier than) the results of Gerlach \cite{Gerlach2013},
but their approach is different, focusing on $L^1$-spaces and employing methods from stochastics. 
Later on, in \cite{Pichor2016, Pichor2018a},
these authors adapted their original results to various situations
involving semigroups
on $L^1$-spaces, with numerous applications in mathematical biology.

\begin{rem}
As noted above, 
the implication ``Theorem \ref{clas.t.greiner}$\dann$Theorem
\ref{clas.t.arendt}'' is almost immediate. Now, 
in hindsight,  it becomes clear that  Arendt in \cite{Arendt2008} was 
also very close to  proving the converse implication
``Theorem \ref{clas.t.arendt}$\dann$Theorem \ref{clas.t.greiner}''.
Indeed, the weak compactness of order intervals in $\Ell{p}$-spaces
for $1\le p < \infty$ implies that on such spaces a positive bounded 
semigroup $T=(T_s)_{s\ge 0}$ 
with a quasi-interior fixed point  is relatively weakly compact. Hence, 
the ``triviality of the peripheral point spectrum'' asserted by
Theorem \ref{clas.t.arendt} implies that $T$ acts trivially
on  the ``reversible'' part of the corresponding JdLG-decomposition.
One can then infer strong convergence of $T$ if one knows
that $T$ is not just relatively weakly, but even relatively {\em strongly}
compact. And the latter holds, in fact, 
since kernel operators are AM-compact; but in this context this
was noted only later by Gerlach and Glück. 
\end{rem}

\subsection*{Old Theorems in a New Light}

Let us now review some of the abovementioned results in the light of 
our actual findings. First of all, consider the following result, 
which is merely an instantiation of Theorem~\ref{conv.t.main1} a), to
the (divisible!) semigroup $\R_+$. 

\begin{thm}\label{clas.t.greiner-new}
Let $T = (T_s)_{s\ge 0}$ be a positive and bounded (but not necessarily strongly
continuous) semigroup on a Banach lattice $E$
with the following
properties:
\begin{aufziii}
\item There is a quasi-interior point $y\in E_+$ and $c > 0$ such
  that $T_s y \le c y$ for all $s\ge 0$.
\item For some $s_0 > 0$ the operator $T_{s_0}$ is AM-compact.
\end{aufziii}
Then $\lim_{s\to \infty} T_s x$ exists for each $x\in E$.
\end{thm}

Theorem \ref{clas.t.greiner-new} is a slight strengthening of  Theorem~4.5 from
\cite{GerlachConvPOS}, where the quasi-interior point $y$ is required 
to be $T$-fixed. It implies
Greiner's Theorem \ref{clas.t.greiner} as a special case:
simply note that on $E:= \Ell{p}(\prX)$ a strictly positive function
is quasi-interior and that a kernel operator is AM-compact (Theorem \ref{int.t.AM}).

\medskip
In the next result, the requirement that one of the semigroup
operators is AM-compact is relaxed towards a mere domination property, however
on the expenses of strengthening other hypotheses.

\begin{thm}\label{clas.t.pichor-new}
Let $T = (T_s)_{s\ge 0}$ be a positive, bounded and irreducible (but not necessarily strongly
continuous) semigroup on a Banach lattice $E$ with order continuous norm. 
Suppose that the conditions are satisfied:
\begin{aufziii}
\item There is a quasi-interior point $y\in E_+$ such
  that $T_s y \le y$ for all $s\ge 0$.
\item For some $s_0 > 0$ there is an AM-compact operator $K\neq 0$ with $0 \le K \le T_{s_0}$. 
\end{aufziii}
Then $\lim_{s\to \infty} T_s x$ exists for each $x\in E$.
\end{thm}

\begin{proof}
By Lemma~\ref{classical.l.dom}, $T$ satisfies the standard assumptions
(see Section~\ref{s.conv}). Hence, as $\R_+$ is a divisible semigroup, 
the assertions follow from Theorem~\ref{conv.t.main1} a).
\end{proof}

Theorem \ref{clas.t.pichor-new} is a generalization of the
abovementioned results \cite[Theorems~1 and~2]{Pichor2000} of Pich\'or
und Rudnicki for stochastic $C_0$-semigroups on $L^1$-spaces.
For $C_0$-semigroups on Banach lattices, the theorem is due to Gerlach \cite[Theorem~4.2]{Gerlach2013}.

We note that irreducibility of the semigroup can be replaced by other assumptions
ensuring that the AM-compact operator $K$ is ``sufficiently large'' when compared
with the semigroup. A very general result of this type was proved by
Gerlach and Gl\"uck in \cite[Theorem~3.11]{GerlachConvPOS}.

\bigskip
Let us finally return to the  spectral-theoretic results (by
Davies, Keicher, Wolff and Greiner--Arendt) discussed above. In this direction, we
establish the following general theorem.

\begin{thm}\label{clas.t.gluha}
Let $T = (T_s)_{s \in S}$ be a bounded and positive
semigroup on a Banach lattice $E$. Suppose that for some
$s\in S$ the operator $T_s$ is AM-compact. 
Then the following assertions hold:
\begin{aufzi}
\item Each unimodular eigenvalue is torsion.
\item If $T$ is irreducible then there are only finitely many unimodular
  eigenvalues.  
 
\item Suppose that  {\rm 1)} $S$ is essentially divisible or {\rm 2)}
  $T$ is strongly continuous
with respect to some topology on $S$ such that the only clopen
  subsemigroup of $S$ containing $0$ is $S$ itself. 
  Then the only possible unimodular eigenvalue of $T$ is the constant one. 
\end{aufzi}
\end{thm}

\begin{proof}
We combine the classical ideas from Scheffold \cite{Scheffold1971} 
as employed by Keicher in \cite[Theorem
3.1]{Keicher2006} with the theory developed in this paper. 

\prfnoi
a)\ Let $\lambda = (\lambda_s)_{s\in S}$ be a unimodular eigenvalue of
$T$, and let $0\neq z \in E$ be a corresponding eigenvector. 
Abbreviate $y := \abs{z} \in E_+$. Then
\[ 0 \neq y = \abs{z} = \abs{\lambda_s z} = \abs{T_sz}\le T_s \abs{z}
= T_s y\qquad
(s\in S).
\]
It follows that the net $(T_sy)_{s\in S}$ is increasing.

\prfnoi
One can find a positive
linear functional $\vphi \in E'_+$ such that $\vphi(y) > 0$. Define
\[ J := \{ x\in E \suchthat \lim_{s\in S} \vphi( T_s\abs{x}) = 0\}.
\]
It is routine to check that  $J$ is a closed and $T$-invariant ideal. 
Since $\vphi(T_sy) \ge \vphi(y) > 0$ for all $s\in  S$, we have $y \notin J$.
Moreover, $T_ty -y \in J$ since
\[ \vphi(T_s\abs{T_t y - y}) = \vphi T_s (T_t y - y) =
\vphi(T_{t{+}s}y)
- \vphi(T_s y) \qquad (s\in S)
\]
and $(\vphi(T_sx))_{s\in S}$ is increasing and bounded. 

\prfnoi
Since $J$ is a closed $T$-invariant ideal, the quotient space
$E_1 := E/J$ naturally carries the structure of a Banach lattice,
and the representation $T$ on $E$ induces a representation $\hat{T}$
on $E_1$ by 
\[  \hat{T}_s(x{+}J) := T_s x + J \qquad  (s\in S,\, x\in E).
\]
Let $\hat{z} := z+J$ and  $\hat{y} := y + J$ be the equivalence classes of $z$ and $y$ in $E_1= E/J$,
respectively.  Since the canonical surjection is a lattice
homomorphism, $\hat{y} = \abs{\hat{z}}$ in $E_1$. Since $y \notin J$,
$\hat{z} \neq 0$.  It follows that $\hat{z}$ is an eigenvector
of $\hat{T}$ for the eigenvalue $\lambda$.  

\prfnoi
Since $T_s y - y\in J$ for each $s\in S$,  
the point $\hat{y}$ is $T$-fixed for the induced semigroup on
$E_1$. Moreover, by the hypothesis and Theorem \ref{AM.t.factor}, for some $s\in S$
the operator $\hat{T}_s$ is AM-compact. Hence, when we restrict to the
closed ideal $E_2 := \cls{F_{\hat{y}}}$ generated by $\hat{y}$ in
$E_1$, 
we find that the semigroup $\hat{T}$ restricted to $E_2$ satisfies the
standard assumptions. Theorem \ref{conv.t.spec} then yields that
that $\lambda$ must be torsion.

\prfnoi
c)\ We start again as in the proof of a). By Theorem
\ref{conv.t.main1}, either condition 1) and 2) implies that 
$\hat{T}$ on $E_2$ is convergent. Then, by Theorem \ref{conv.t.spec}, we conclude that
$\lambda$ is constant. 
\end{proof}

Theorem \ref{clas.t.gluha} generalizes the Greiner--Arendt  Theorem
\ref{clas.t.arendt}: simply specialize $S = \R_+$ and note that kernel
operators are AM-compact (Theoren \ref{int.t.AM}). A fortiori, it
generalizes Davies' results  from \cite{Davies2005}. However, it also 
implies Keicher's result \cite[Theorem ~3.1]{Keicher2006} (but not
Wolff's), as on  an atomic Banach lattice with order-continuous norm all
order intervals are relatively compact, and hence all bounded operators are
AM-compact. Finally, Theorem \ref{clas.t.gluha} also generalizes 
\cite[Theorem~4.19]{GerlachConvPOS}.

\appendix

\section{AM-Compact Operators}\label{app.int}

Let $E$ be a Banach lattice and $F$ a Banach space. A bounded operator
$T: E \to F$ is called {\emdf AM-compact} if $T$ maps order intervals
of $E$ to relatively compact subsets of $F$.

\medskip
The following is a  useful characterization of AM-compactness
in the case that $E$ has a  quasi-interior point $y\in E_+$. 
Recall that this means that the principal ideal
\[    E_y := \{ x\in E \suchthat \text{there is $c \ge 0$ such that 
$\abs{x}\le cy$}\}
\]
is dense in $E$. We endow $E_y$ with its natural  AM-norm
\[ \norm{x}_y := \inf\{
 c \ge 0 \suchthat \abs{x}\le cy\}.
\]
It is well known that this turns $E_y$ into a Banach lattice,
isometrically lattice isomorphic to
$\Ce(K)$ (with $y$ being mapped to $\car$) for some compact Hausdorff
space $K$ (Krein--Krein--Kakutani theorem).

\begin{lem}\label{AM.l.AM-char}
Let $E$ be a Banach lattice with a quasi-interior point $y \in E_+$,
let $F$ be any other Banach space and $T: E \to F$ a bounded operator.
Then $T: E_y  \to F$ is compact if and only if $T[0,y]$ is relatively
compact, if and only if $T$ is AM-compact.
\end{lem}

\begin{proof}
Suppose that $T: E_y \to F$ is compact. Then, $T[0,y]$ is relatively compact.
If the latter is the case, then for each 
$c> 0$ the set
\[ T[-cy, cy] = T( -cy + 2c [0, y]) = T(-cy) + 2c T[0,y] 
\]
is also relatively compact. Let $u \in E_+$ and $\veps > 0$. 
Since $y$ is a quasi-interior point, there is $c> 0$ such that
$u \in [-cy, cy] + \Ball[0, \veps]$. We claim that
\beq\label{AM.e.AM-char}   [0, u] \subseteq [-cy, cy] + \Ball[0, \veps]
\eeq
as well. Indeed, write $u = z + r$ with
$\abs{z}\le cy$ and $\norm{r}\le \veps$ and let $0 \le x \le u$. 
Then $0 \le x \le \abs{z} + \abs{r} \le cy + \abs{r}$. By the
decomposition
property, there are $0 \le x_1 \le \abs{z}$ and $0 \le x_2\le \abs{r}$
with $x = x_1 + x_2$. Hence $x\in [-cy, cy] + \Ball[0, \veps]$ as
claimed.

It follows from \eqref{AM.e.AM-char} that
\[ T[0,u] \subseteq T[-cy, cy] + \Ball[0, \norm{T} \veps].
\]
Since $T[-cy, cy]$ is relatively compact, it admits a finite
$\norm{T} \veps$-mesh. Hence, $T[0,u]$ admits a finite $2 \norm{T} \veps$-mesh. 
As $\veps > 0$ was arbitrary, $T[0,u]$ is relatively compact. 

Finally, if $[u,v]$ is any non-empty order interval of $E$, then $v-u \ge 0$ and
$T[u,v] = Tu + T[0, v-u]$ is relatively compact, by what we have
already shown.
\end{proof}

The following theorem shows that AM-compactness is preserved when one
passes to a factor lattice with respect to an invariant closed ideal.

\begin{thm}\label{AM.t.factor}
Let $E$ be a Banach lattice, let $T\in \BL(E)$ be AM-compact, and let
$I\subseteq E$ be a $T$-invariant closed lattice ideal in $E$. 
Then the induced operator 
\[ T_/: E/I \to E/I \qquad T_/(x + I) := Tx + I \qquad (x\in E)
\]
is also AM-compact.
\end{thm}

\begin{proof}
Let $0 \le X \in E/I$. We have to show: the order interval $[0, X]$ in
$E/I$ is mapped by $T_/$ to a relatively compact subset of $E/I$.

Recall from \cite[Prop.~II.2.6]{Schaefer1974} that for $x, y\in E$ one has
\[ x{+}I \le y{+}I \quad \iff\quad \exists x_1, y_1 \in E :\, x-x_1, y-y_1\in
I,\,\, x_1\le y_1.
\]
In particular, there is $x\ge 0$ such that $X =
x{+}I$. Now, let $0 \le Y \le X$ be given. Then there are $y\in E$ and
$z\in I$ such that $y \le x + z$ and $Y = y{+}I$. Replacing $z$ by $\abs{z}$
we may suppose in addition that $z\ge 0$. 

By the decomposition property, $y^+ = a+ b$ for some elements
$0 \le a \le x$ and $0 \le b \le z$. Then $b \in I$ since $I$ is an
ideal and $z\in I$. As the canonical surjection is a lattice
homomorphism, $Y^+ = a{+}I$ and hence $T_/(Y^+) = Ta{+}I$. Similarly,
one can show that there is $c\in [0,x]$ with $Y^- = c{+}I$. It follows
that 
\[ T_/(Y) = T_/(Y^+) - T_/(Y^-) \in T([-x,x]) + I
\]
and hence $T_/([0, X]) \subseteq T([-x,x]) + I$. Since  $T$ is
AM-compact, the set $T([-x,x])$ is relatively compact in $E$, and hence
so is $T([-x,x]) + I$ in $E/I$. This implies that 
$T_/([0, X])$ is relatively compact, as desired.
\end{proof}

\subsection*{Examples of AM-Compact Operators}

In the remainder of this appendix we consider two results that help to
identify AM-compact operators. The first tells that every integral
operator is AM-compact. This is a well-known consequence
of abstract theory (see e.g.\ the discussion at the beginning of
\cite[Section~4]{GerlachConvPOS}, which is based on abstract results in
\cite[Corollary~3.7.3]{Meyer-Nieberg1991}
and \cite[Proposition~IV.9.8]{Schaefer1974}). 
Gerlach and Gl\"uck give an elementary proof involving measure-theoretic
(i.e., almost everywhere) arguments \cite[Proposition~A.1]{GerlachConvPOS}.
Our proof is a little less elementary, but replaces
the measure theory by functional analysis.

\begin{thm}\label{int.t.AM}
Let $\prX$ and $\prY$ be measure spaces and let
$1 \le p, q < \infty$. Let, furthermore, 
$k: X \times Y \to \R_+$ be a measurable mapping such that 
by
\[ Tf (x) := \int_\prY k(x, \cdot) f\qquad (x\in X)
\]
a bounded operator
\[ T: \Ell{p}(\prY) \to \Ell{q}(\prX)
\]
is (well) defined. Then $T$ is AM-compact.
\end{thm}

\begin{proof}
Let the measure spaces be $\prX = (X, \SigmaX, \muX)$ and $
\prY = (Y, \SigmaY,\muY)$.
Let $0 \le u \in \Ell{p}(\prY)$. We have to show that 
$T[0,u] \subseteq \Ell{q}(\prX)$ is relatively compact.

\smallskip
\noindent
Suppose first that both measures $\muX$ and $\muY$ are finite, 
$u = \car$ and $T\car \in \Ell{\infty}(\prX)$.
Then, in particular,  $k \in \Ell{1}(\prX \times \prY)$. Since
$\Ell{1}(\prX) \tensor \Ell{1}(\prY)$ is dense in 
$\Ell{1}(\prX\times \prY)$, $T: \Ell{\infty}(\prY) \to \Ell{1}(\prX)$
is compact. But $T$ factors through $\Ell{\infty}(\prX)$ (since
$T\car$ is bounded) and since on closed $\Ell{\infty}$-balls
the $\Ell{q}$- and the $\Ell{1}$-topology
coincide, $T: \Ell{\infty}(\prY) \to \Ell{q}(\prX)$ 
is compact. In particular, $T[0,\car]$ is relatively
compact in $\Ell{q}(\prX)$.

\smallskip
\noindent
In the general case, consider the 
operator
\[ S: \Ell{p}(Y,\SigmaY,u^p\muY) \to \Ell{q}(\prX),\qquad 
Sf := T(uf).
\]
Because of $[0, u] = u \cdot [0, \car]$ it suffices to show that 
$S[0,\car] \subseteq \Ell{q}(\prX)$ is relatively compact. Hence, one
may suppose without loss of generality  that $\muY$ is finite and 
$u = \car$. 

\smallskip\noindent
Under this assumption let
\[ g = T\car:= \Bigl(x \mapsto \int_Y k(x,y)\, \muY(\ud{y})\Bigr) \in \Ell{q}(\prX)
\]
and, for $n \in \N$, 
\[ g_n := \car_{\set{\frac{1}{n} \le g \le n}} g \in
\Ell{1}(\prX)\cap\Ell{\infty}(\prX).
\]
Define $T_n$ by 
\[ T_n f := \car_{\set{\frac{1}{n} \le g \le n}} Tf = 
\int_Y \car_{\set{\frac{1}{n} \le g \le n}}(x) \,  k(x,y)f(y)
\, \muY(\ud{y}). 
\]
Observe that  $T_n \car \in \Ell{\infty}(\prX)$ and it is supported
on a set of finite measure. So, by what we have shown first,
$T_n: \Ell{\infty}(\prY) \to \Ell{q}(\prX)$ is compact. 
Now, 
\[ \abs{(T- T_n)f} \le \Bigl( \car_{\set{g \le \frac{1}{n}}}  + 
\car_{\set{g \ge n}}\Bigr)\, g \,  \norm{f}_{\Ell{\infty}}
\]
and hence
\[ \norm{T-T_n}_{\Ell{q}\leftarrow\Ell{\infty}} \le 
\norm{\dots g}_{\Ell{q}}  \to 0
\qquad (n \to \infty).
\]
Consequently, also $T: \Ell{\infty}(\prY) \to \Ell{q}(\prX)$
is compact, and this concludes the proof.
\end{proof}

We turn to a second class of examples of AM-compact operators.

\begin{thm}\label{int.t.fact-Linfty}
Let $\prX$ and $\prY$ be measure spaces, let
$1 \le p, q < \infty$, and let
\[ T: \Ell{p}(\prY) \to \Ell{q}(\prX)
\]
be a positive operator with the following property:
There is a sequence $(A_n)_n$ in $\Sigma_\prX$ such that
$\bigcup_n A_n = X$ and such that 
\[ \car_{A_n} \cdot \ran(T) \subseteq \Ell{\infty}(A_n) \qquad
\text{for all $n\in \N$}.
\]
Then $T$ is AM-compact.
\end{thm}

\begin{proof}
By performing the same reduction as in the proof of Theorem \ref{int.t.AM}
above, we may suppose without loss of generality that $\prY$ is a
finite measure space. It remains to show that $T[0,\car]$ is
relatively $\Ell{q}$-compact.

Define the projection $P_n$ on $\Ell{q}(\prX)$ by 
\[ P_n f :=
\car_{A_n \cap \set{T\car \ge \frac{1}{n}}} f.
\]  
By hypothesis and the dominated convergence theorem,  
$P_nT \to T$ uniformly on order intervals of
$\Ell{p}(\prY)$, hence in particular with respect to the norm of $\BL(\Ell{\infty};
\Ell{q})$. Therefore, it suffices to consider the case that also
$\prX$ is a finite measure space, and $\ran(T) \subseteq
\Ell{\infty}(\prX)$.  By the subsequent Lemma \ref{int.l.weak-strong} we are done.
\end{proof}

\begin{lem}\label{int.l.weak-strong}
Let $\prX$ and $\prY$ be finite measure spaces, $1\le p, q < \infty$
and $T: \Ell{p}(\prY) \to \Ell{q}(X)$ a bounded operator
with $\ran(T) \subseteq \Ell{\infty}(\prX)$. Then for each 
$1 < r \le \infty$ with $r\ge p$ the operator 
\[ T\res{\Ell{r}} : \Ell{r}(\prY) \to \Ell{q}(\prX)
\]
is compact.
\end{lem}

\begin{proof}
We may suppose $1 < p = r < \infty$. By the closed graph theorem, 
$T: \Ell{p}(\prY) \to \Ell{\infty}(\prX)$ is bounded. If $(f_n)_n$ is a bounded
sequence
in $\Ell{p}(\prY)$, it has a weakly convergent subsequence (by reflexivity).
Hence, $(Tf_n)_n$ has a weakly convergent subsequence in
$\Ell{\infty}(\prX)$.  But a weakly convergent sequence in
$\Ell{\infty}(\prX)$
converges strongly in $\Ell{q}(\prX)$ since $\prX$ is a finite measure
space.
(Represent $\Ell{\infty}(\prX) \cong \Ce(K)$ for some compact
Hausdorff space $K$, and observe that in $\Ce(K)$ a sequence(!) is
weakly
convergent if and only if it is uniformly bounded and pointwise convergent.)
\end{proof}

\begin{rems}
\begin{aufziii}
\item Actually, one can reduce Theorem \ref{int.t.fact-Linfty} to 
Theorem \ref{int.t.AM} by employing Bukhvalov's characterization of
integral operators \cite{Bukhvalov1978}, variants of which are for instance
discussed in \cite[Theorem~3.9]{Grobler1980/81}, \cite[Theorem~96.5]{Zaanen1983}, 
\cite[Theorem~1.5]{Arendt1994} 
and \cite[Theorem~4.2.12]{Gerlach2014a}

\item For $1\le p \le 2$, Lemma \ref{int.l.weak-strong} has a much
  more elementary proof. Indeed, by quite elementary arguments one can
  show that $T$ is a Hilbert--Schmidt operator when considered as an
  operator $\Ell{2}(\prY) \to \Ell{2}(\prX)$. 

\item In the proof of  Lemma \ref{int.l.weak-strong} we have
used that, for a finite measure space $\prX$, a sequence which 
converges weakly in $\Ell{\infty}(\prX)$ must converge strongly
on $\Ell{q}(\prX)$ for each $1\le q < \infty$. It would be nice to 
have a more elementary proof of this fact, avoiding the representation
$\Ell{\infty}(\prX)\cong \Ce(K)$.  

\end{aufziii}
\end{rems}

\section{Universal Nets}\label{app.uni}

In this appendix we treat, for the convenience of the reader, a
useful but maybe not so widely known concept from general topology.

\medskip
\noindent

Let $(I, \le)$ be any directed set. A subset $A$ of $I$ is called
a {\emdf tail} of $I$ if it is of the form
$A =\{ \alpha \in I \suchthat \alpha \ge \alpha_0\}$ for some
$\alpha_0 \in I$. And it is called {\emdf cofinal} if its complement
does not contain a tail. Equivalently, $A$ is cofinal if 
for each $\alpha_0\in I$ there is $\alpha \ge
\alpha_0$ such that $\alpha \in A$. Clearly, each tail is  cofinal.

A {\emdf tail} of a net $(x_\alpha)_{\alpha \in I}$ in a set $X$ is 
a subset of the form $\{x_\alpha \suchthat \alpha \in A\}$, where
$A\subseteq I$ is a tail of $I$. A net $(x_\alpha)_\alpha$
is called {\emdf
  universal} or an {\emdf ultranet}, if for each $Y\subseteq X$ either  the set
$Y$ or its complement contains a tail of the net. If $f: X_1 \to X_2$ is a
mapping
and $(x_\alpha)_{\alpha}$ is a universal net in $X_1$, then
$(f(x_\alpha))_{\alpha}$ is a universal net in $X_2$. Likewise, a
subnet of a universal net is universal.

\begin{lem}\label{uni.l.kelley}
Each net has a universal subnet.
\end{lem}

\begin{proof}
Let $(x_\alpha)_{\alpha}$ be a net in a set $X$. Then the tails of the
net form a filter base, hence by Zorn's lemma there is an ultrafilter
 $\calU$ containing all the tails. For $\beta \in I$ and $U \in
 \calU$ there is $\alpha(\beta,U) \in I$ such that
$\alpha(\beta,U) \ge \beta$ and $x_{\alpha(\beta, U)} \in U$. The set
$I \times \calU$ is directed by 
\[ (\alpha, U)\le (\beta, V) \quad\defiff\quad \alpha \le
\beta\quad\wedge\quad
V \subseteq U.
\]
The net $(x_{\alpha(\beta,U)})_{\beta,U}$ is a universal subnet of $(x_\alpha)_\alpha$.
\end{proof}

Since a universal net in a
topological space $X$ converges
if and only if it has a convergent subnet,  $X$ is
compact if and only if each universal net in $X$ converges.

\begin{lem}\label{uni.l.cauchy}
Let $X$ be a  metric space and let $(x_\alpha)_{\alpha \in I}$ be a 
net in $X$. Consider the following three assertions:
\begin{aufzii}
\item For each $\veps > 0$ there is a compact set $K \subseteq X$ such
  that $\Ball[K,\veps]$ contains a tail of $(x_\alpha)_\alpha$.

\item For each $\veps > 0$ there is $z \in X$ such that
$\{ \alpha \suchthat x_\alpha \in \Ball[z,\veps]\}$ is cofinal.

\item The net $(x_\alpha)_{\alpha}$ is a Cauchy net.
 
\end{aufzii}
Then {\rm (iii)}$\dann${\rm (i)}$\dann${\rm (ii)}. And if
$(x_\alpha)_{\alpha\in I}$ is universal, also {\rm (ii)}$\dann${\rm
  (iii)}.  (Hence, in the latter case, all three assertions are
equivalent.)
\end{lem}

\begin{proof}
(iii)$\dann$(i): Let $\veps>0$. Then there is $\alpha \in I$ such
that $d(x_\beta, x_\gamma)\le \veps$ for all $\beta, \gamma\ge
\alpha$. Hence $\Ball[\{x_\alpha\}, \veps]$ contains the tail
$\{ x_\beta \suchthat \beta \ge \alpha\}$.

\smallskip
\noindent
(i)$\dann$(ii):\ Let $\veps > \veps'> 0$. Then there is a finite subset
$F\subseteq K$ such that $K \subseteq \bigcup_{z\in F} \Ball(z,
\veps')$. Hence, by (ii), the set $\bigcup_{z\in F} \Ball[z, 2 \veps]$
contains a tail of $(x_\alpha)_{\alpha}$. Since $F$ is finite,
there is $z\in F$ such that $\{ \alpha \suchthat x_\alpha \in \Ball[z,
2\veps]\}$ is cofinal.

\smallskip
\noindent
(ii)$\dann$(iii):\  Suppose that $(x_\alpha)_\alpha$ is universal, 
let $\veps > 0$ and pick $z$ as in (ii).
By the universality of the net either the set
$A_\veps := \{ \alpha \suchthat x_\alpha \in \Ball[z, \veps]\}$ or its 
complement contains a tail of $I$.  But since $A_\veps$ is cofinal,
the second alternative is impossible. 
Hence, $d(x_\alpha, x_\beta) \le 2 \veps$ for all $\alpha,
\beta$ from a tail of $I$. 
\end{proof}

\begin{thm}\label{uni.t.cluster}
Let $(x_\alpha)_{\alpha\in I}$ be a net in a regular topological space
$X$ and let 
\[ C := \bigcap_{\beta \in I} \cls{\{ x_\alpha \suchthat \alpha \ge
  \beta\}}
\]
be its set of cluster points. Consider the following assertions.
\begin{aufzii}
\item Each subnet of $(x_\alpha)_{\alpha\in I}$ has a cluster point.
\item Each universal subnet of $(x_\alpha)_{\alpha\in I}$ converges.
\item For each cofinal subsequence $(\alpha_n)_{n \in \N}$ the
  sequence $(x_{\alpha_n})_{n \in \N}$ has a  cluster point.

\item The set $C$ is non-empty and compact. 
\end{aufzii}
Then {\rm (i)}$\gdw${\rm (ii)}$\dann${\rm (iii)} and {\rm (i)}$\dann${\rm (iv)}. 
If, in addition, $I$ admits cofinal subsequences and $X$ is
metrizable, then {\rm (iii)}$\dann${\rm (iv)} as well.
\end{thm}

\begin{proof}
(i)$\gdw$(ii): This follows from Lemma \ref{uni.l.kelley} and the fact that a
universal net converges if and only if it has a cluster point.

\prfnoi
(i)$\dann$(iv): By hypothesis, the net $(x_\alpha)_{\alpha\in I}$
itself has a cluster point, so $C\neq \leer$.  Let $(y_j)_{j\in J}$ be
a universal net in $C$. It suffices to show that $(y_j)_j$ converges. Let
\[ M := \{ (\alpha, j, U) \suchthat \text{$\alpha\in I$,\, $U$ open in $X$ and $y_k \in U$
  for all $k \ge j$} \}
\]
be directed by 
\[  (\alpha, j_0, U) \le (\beta, j_1, V) \quad\defiff\quad \alpha \le
\beta\,\, \wedge \,\, j_0 \le j_1 \,\, \wedge\,\,
V \subseteq U.
\]
For each $(\alpha, j, U) \in M$ one has $y_j \in U$ and hence there is
\[ \alpha \le \vphi(\alpha, j, U) \in I \quad \text{with}\quad 
x_{\vphi(\alpha, j, U)} \in U.
\]
Then the mapping $\vphi: M \to I$ is cofinal, so $(x_{\vphi(\alpha, j,
  U)})_{(\alpha, j, U) \in M}$ is a subnet. By hypothesis (i), this
subnet has a cluster point $y\in X$, say. We prove that $y_j \to y$. 

To this end, let $V$ be any open neighborhood of $y$ in $X$. Since $X$
is regular, there is an open set $W$ such that $y\in W \subseteq
\cls{W} \subseteq V$. It suffices to show that $\cls{W}$ contains a
tail of $(y_j)_j$. Suppose that this is not the case. Then, since
that net is universal, the complement $U_0 := X \ohne \cls{W}$
contains tail of $(y_j)_j$. This means that there is $j_0\in J$ such
that  $y_j \in U_0$ for all $j \ge j_0$.  In particular,
$(\alpha, j_0, U_0) \in M$ for all $\alpha \in I$. 

Fix $\alpha_0 \in I$. Then, since $y$ is a cluster point of
$(x_{\vphi(m)})_{m \in M}$  and $W$ is an open neighborhood of $y$, 
there is $M \ni (\alpha, j, U) \ge
(\alpha_0, j_0, U_0)$ with $x_{\vphi(\alpha,j, U)} \in W$. But by the
construction of $\vphi$ we have also
$x_{\vphi(\alpha,j, U)} \in U \subseteq U_0$. Since $U_0 \cap W =
\leer$, this yields a contradiction and the proof is complete.

\prfnoi
(i)$\dann$(iii): This is trivial, as the sequence $(x_{\alpha_n})_n$
is a subnet of $(x_\alpha)_\alpha$ whenever 
$(\alpha_n)_n$ is a cofinal
sequence in $I$.

\prfnoi
(iii)$\dann$(iv): Suppose that $d$ is a metric inducing the
topology of $X$ and that $I$ admits cofinal sequences. It then
follows from (iii) that $C\neq \leer$. In order to see that $C$ is
compact, let $(y_n)_n$ be a sequence in $C$ and let 
$(\alpha_n)_n$ be a cofinal sequence in $I$. Recursively, one can find
a sequence $(x_{\beta_n})_n$ such that
$\beta_n \ge \alpha_n$ and $d(x_{\beta_n}, y_n) \le \frac{1}{n}$. 
By (iii), $(x_{\beta_n})_n$ has a cluster point. Since $X$ is metric,
this means that $(x_{\beta_n})_n$ has a convergent subsequence. 
But then $(y_n)_n$ also has convergent subsequence, and as $C$ is
closed, the limit of this subsequence lies in $C$.  
\end{proof}

\section{Some Notions from Semigroup Theory}\label{app.sgp}

A {\emdf semigroup} is a  non-empty set $S$ together with an
associative operation $S\times S \to S$, generically called
``multiplication''. A subset $M$  of a semigroup $S$ is {\emdf
  multiplicative}, if $M \cdot M \subseteq M$. A non-empty
multiplicative subset is a  {\emdf subsemigroup}.
A non-empty subset $J$ of a semigroup $S$ is a (two-sided) {\emdf ideal} of $S$ 
if $SJ \cup J S \subseteq J$. Each ideal is a subsemigroup.

A {\emdf neutral element} in a semigroup is any element $e\in S$ such
that $e s = se = s$ for all $s\in S$. There is at most
one neutral element; if there is none, one can adjoin one
in a standard way.   

A semigroup $S$ is called {\emdf Abelian} or {\emdf commutative} if 
$st =ts$ for all $s, t\in S$. It is common to write 
Abelian semigroups additively, and denote their neutral elements
by $0$.

\medskip

An Abelian semigroup $S$ is called {\emdf cancellative} if
it satisfies the implication
\[  s + t = s + t'\quad \dann\quad t = t'
\]
for all $s,t, t'\in S$. By a standard result from semigroup theory,
$S$ is faithfully embeddable into a group
if and only if $S$ is cancellative.  In this case, one usually
considers $S$ as a {\em subset} of some group $G$ such that $G =
S{-}S$. The (necessarily Abelian) group $G$ 
is then unique up to canonical isomorphism
and is called the group {\emdf generated} by $S$. In each case, when
we speak of a group generated by an Abelian semigroup $S$, 
we have this meaning in mind, and in particular suppose tacitly that 
$S$ is cancellative.

\medskip

An Abelian semigroup $S$ with neutral element $0$ 
is called {\emdf divisible}, if for each $s\in S$ and $n \in
\N$ there is $t\in S$ such that $nt = s$.  (This definition
extends the common notion of divisibility from groups to semigroups.)
And $S$  is called {\emdf essentially divisible} if for each $s\in S$ and
$n \in \N$ there are $t_1, t_2\in S$ such that $nt_1 = s + nt_2$.

Each divisible semigroup and each semigroup that generates a 
divisible group is essentially divisible.

\end{document}

%% file: depmakro.tex
\newcommand{\Sym}{\mathrm{Sym}}

\def\beq{\begin{equation}}
\def\eeq{\end{equation}}

\newcommand{\hs}{\hskip-0.1em}

\def\prfnoi{\smallskip\noindent}
\newcommand{\emdf}{\bf}               % Definitionen im Text

\def\qedsymbol{\hbox to 1ex{\llap{\rule{0.25pt}{1ex}}\rlap{\rule{1ex}{0.25pt}}\lower0.25pt\rlap{\raise1ex\rlap{\rule{1ex}{0.25pt}}}\hskip1ex\llap{\rule{0.25pt}{1ex}}}}
\def\rlqed{\rlap{\rule{\hsize}{0pt}\kern-1ex\kern-1em\qed}} %% \qed rechts

\newcommand{\beproof}[1]{%
\renewcommand{\proofname}{#1}\begin{proof}}
\newcommand{\enproof}{%
\end{proof}\renewcommand{\proofname}{Proof}}

\newcommand{\wot}{\mathrm{w}} %% weak op top
\newcommand{\sot}{\mathrm{s}} %% strong op top
\newcommand{\dann}{\,\Rightarrow\,}% \rightarrow
\newcommand{\gdw}{\,\Leftrightarrow\,}
\newcommand{\Iff}{\Longleftrightarrow}

\newcommand{\defiff}{\stackrel{\text{\upshape\tiny def.}}{\Iff}}

\newcommand{\ohne}{\setminus}
\newcommand{\leer}{\emptyset}
\newcommand{\set}[1]{\hs\left[\,#1\,\right]}

\newcommand{\suchthat}{\,\,|\,\,}     % Durch Aussonderung definierte Mengen
\DeclareMathOperator{\dom}{dom}
\DeclareMathOperator{\ran}{ran}
\DeclareMathOperator{\fix}{fix}
\newcommand{\res}[1]{|_{#1}}
\newcommand{\car}{\mathbf{1}}           % Konstante Funktion 1
\def\id{\mathop{\mathrm{id}}\nolimits}

\newcommand{\R}{\mathbb{R}}
\newcommand{\N}{\mathbb{N}}
\newcommand{\Z}{\mathbb{Z}}

\newcommand{\C}{\mathbb{C}}

\newcommand{\K}{\mathbb{K}}

\newcommand{\abs}[1]{\left\vert#1\right\vert}   % Absolutbetrag math. Modus
\newcommand{\cls}[1]{\overline{#1}}     % Abschluss
\newcommand{\Ball}{\mathrm{B}}          % Ball(x,r)
\newcommand{\cl}{\mathrm{cl}}

\DeclareMathOperator{\dist}{dist}       % Abstand
\DeclareMathOperator{\spann}{span}  

\newcommand{\tensor}{\otimes}

\newcommand{\norm}[1]{\| #1 \| }

\makeatletter
\newcommand{\tdprod}[2]{\ensuremath{%
  \setbox0=\hbox{\ensuremath{\langle#1,#2\rangle}}
  \dimen@\ht0
  \advance\dimen@ by \dp0 \langle #1\rule[-\dp0]{0pt}{\dimen@}\,, #2\hspace{1pt}\rangle}}

\newcommand{\bdprod}[2]{\ensuremath{%
  \setbox0=\hbox{\ensuremath{\bigl\langle #1,#2 \bigr\rangle)}}
  \dimen@\ht0
  \advance\dimen@ by \dp0 \bigl\langle#1\bigl.,\rule[-\dp0]{0pt}{\dimen@}\bigr.#2\hspace{1pt}\bigr\rangle}}
\newcommand{\Bdprod}[2]{\ensuremath{%
  \setbox0=\hbox{\ensuremath{\Bigl(#1,#2\Bigr)}}
  \dimen@\ht0
  \advance\dimen@ by \dp0 \Bigl(#1\Bigl|\rule[-\dp0]{0pt}{\dimen@}\Bigr.#2\hspace{1pt}\Bigr)}}
  
 \makeatother

\DeclareMathOperator{\BL}{\calL}
\def\Id{{\mathop{\mathrm{I\mathstrut}}}}
\newcommand{\spec}{\usigma}

\newcommand{\Pspec}{\spec_{\mathrm{p}}}

\newcommand{\CO}{\mathcal{K}}    % compact operators
\makeatletter
\newcommand{\tsprod}[2]{\ensuremath{%
  \setbox0=\hbox{\ensuremath{(#1,#2)}}
  \dimen@\ht0
  \advance\dimen@ by \dp0 (#1\rule[-\dp0]{0pt}{\dimen@}\,|#2\hspace{1pt})}}

\newcommand{\sprod}[2]{\ensuremath{%
  \setbox0=\hbox{\ensuremath{\left(#1,#2\right)}}
  \dimen@\ht0
  \advance\dimen@ by \dp0 \left(\left.#1\rule[-\dp0]{0pt}{\dimen@}\,\right|#2\hspace{1pt}\right)}}

\newcommand{\bsprod}[2]{\ensuremath{%
  \setbox0=\hbox{\ensuremath{\bigl(#1,#2\bigr)}}
  \dimen@\ht0
  \advance\dimen@ by \dp0 \bigl(#1\bigl|\rule[-\dp0]{0pt}{\dimen@}\bigr.#2\hspace{1pt}\bigr)}}
\newcommand{\Bsprod}[2]{\ensuremath{%
  \setbox0=\hbox{\ensuremath{\Bigl(#1,#2\Bigr)}}
  \dimen@\ht0
  \advance\dimen@ by \dp0 \Bigl(#1\Bigl|\rule[-\dp0]{0pt}{\dimen@}\Bigr.#2\hspace{1pt}\Bigr)}}
\makeatother

\makeatletter
\newcommand{\tsdprod}[2]{\ensuremath{%
  \setbox0=\hbox{\ensuremath{(#1,#2)}}
  \dimen@\ht0
  \advance\dimen@ by \dp0 \langle#1\rule[-\dp0]{0pt}{\dimen@}\,|#2\hspace{1pt}\rangle}}

\newcommand{\sdprod}[2]{\ensuremath{%
  \setbox0=\hbox{\ensuremath{\left(#1,#2\right)}}
  \dimen@\ht0
  \advance\dimen@ by \dp0 \left\langle\left.#1\rule[-\dp0]{0pt}{\dimen@}\,\right|#2\hspace{1pt}\right\rangle}}

\newcommand{\bsdprod}[2]{\ensuremath{%
  \setbox0=\hbox{\ensuremath{\bigl(#1,#2\bigr)}}
  \dimen@\ht0
  \advance\dimen@ by \dp0 \bigl\langle#1\bigl|\rule[-\dp0]{0pt}{\dimen@}\bigr.#2\hspace{1pt}\bigr\rangle}}
\newcommand{\Bsdprod}[2]{\ensuremath{%
  \setbox0=\hbox{\ensuremath{\Bigl(#1,#2\Bigr)}}
  \dimen@\ht0
  \advance\dimen@ by \dp0 \Bigl\langle#1\Bigl|\rule[-\dp0]{0pt}{\dimen@}\Bigr.#2\hspace{1pt}\Bigr\rangle}}
\makeatother

\newcommand{\Ce}{\mathrm{C}}

\newcommand{\co}{\mathrm{c}_{\mathrm{0}}}

\newcommand{\Ell}[2][\relax]{%     % Ell-p spaces
   \ifx#1\relax \mathrm{L}^{\mathrm{#2}}
   \else \mathrm{L}^{\mathrm{#2}}_{\mathrm{#1}}
   \fi}
\renewcommand{\Ell}[2][\relax]{%    ??????????????????????????????????????????
   \ifx#1\relax \mathrm{L}^{#2}
   \else \mathrm{L}^{#2}_{\mathrm{#1}}
   \fi}

\newcommand{\Wee}[2][\relax]{%    %   Sobolev spaces
   \ifx#1\relax \mathrm{W}^{\mathrm{#2}}
   \else \mathrm{W}^{\mathrm{#2}}_{\mathrm{#1}}
   \fi}
\newcommand{\Har}[2][\relax]{%       %Hardy spaces
   \ifx#1\relax \mathsf{H}^{\mathsf{#2}}
   \else   \mathsf{H}^{\mathsf{#2}}_{\mathrm{#1}}
   \fi}
\newcommand{\prX}{\mathrm X}     % probability space
\newcommand{\prY}{\mathrm Y}

\newcommand{\SigmaX}{\Sigma_\prX}   %the corresponding sigma-algebra
\newcommand{\SigmaY}{\Sigma_\prY}

\newcommand{\muX}{\mu_\prX}     % the P-measure of the probability space \prX
\newcommand{\muY}{\mu_\prY}

\makeatletter

\def\maketag@@@#1{\m@th\normalfont#1%
\gdef\tagform@##1{\maketag@@@{(\ignorespaces##1\unskip\@@italiccorr)}}}
\def\eqtext#1{\global\eqtext@true\gdef\tagform@##1{\maketag@@@{\ignorespaces##1\unskip\@@italiccorr}}\tag{#1}}%
\def\reqtext#1{\gdef\tagform@##1{\maketag@@@{\ignorespaces##1\unskip\@@italiccorr}}\tag{#1}}%

\def\@@@lefttag{
   \kern-\tagshift@
        \if1\shift@tag\row@\relax
            \rlap{\vbox{%
                \normalbaselines
                \boxz@
                \vbox to\lineht@{}%
                \raise@tag
            }}%
        \else
\hskip-0.30em%
\rlap{\boxz@}%
        \fi
        \kern\displaywidth@
    }

\newif\ifeqtext@
\def\place@tag{%
\ifeqtext@
\@@@lefttag
\else
   \iftagsleft@
     \@@@lefttag
     \else
        \kern-\tagshift@
        \if1\shift@tag\row@\relax
            \llap{\vtop{%
                \raise@tag
                \normalbaselines
                \setbox\@ne\null
                \dp\@ne\lineht@
                \box\@ne
                \boxz@
            }}%
        \else
            \llap{\boxz@}%
        \fi
    \fi
   \fi
    \global\eqtext@false
}

\makeatother

\makeatletter
\def\romref#1{(\@roman{\ref{#1}})} %% roman nummer als referenz zu "item"
\makeatother